\documentclass[journal]{IEEEtran}

\IEEEoverridecommandlockouts                              

\usepackage{graphicx}
\usepackage{mathptmx} 
\usepackage{times} 
\usepackage{amsmath} 
\usepackage{amsthm}
\usepackage{amssymb}  
\usepackage{amsfonts}

\usepackage{bm}

\usepackage[utf8]{inputenc}

\usepackage[noadjust]{cite}

\newtheorem{theorem}{Theorem}[section]

\newtheorem{corollary}[theorem]{Corollary}

\theoremstyle{remark}
\newtheorem{remark}[theorem]{Remark}

\theoremstyle{definition}
\newtheorem{definition}{Definition}

\title{\LARGE \bf 
On the Relation between the Minimum Principle and Dynamic Programming for Classical and Hybrid Control Systems}

\author{Ali Pakniyat, \textit{Member, IEEE}, and Peter E. Caines, \textit{Life Fellow, IEEE}

\thanks{This work is supported by the Natural Sciences and Engineering Research Council of Canada (NSERC) and the Automotive Partnership Canada (APC).}
\thanks{Ali Pakniyat and Peter E. Caines are with the Centre for Intelligent Machines (CIM) and the Department of Electrical and Computer Engineering, McGill University, Montreal, QC, Canada \newline
        {\tt\small pakniyat@cim.mcgill.ca, peterc@cim.mcgill.ca}}%
}

\begin{document}

\bstctlcite{IEEEexample:BSTcontrol}

\maketitle
\thispagestyle{empty}
\pagestyle{empty}



\begin{abstract}

Hybrid optimal control problems are studied for a general class of hybrid systems where autonomous and controlled state jumps are allowed at the switching instants and in addition to terminal and running costs switching between discrete states incurs costs. The statements of the Hybrid Minimum Principle and Hybrid Dynamic Programming are presented in this framework and it is shown that under certain assumptions the adjoint process in the Hybrid Minimum Principle and the gradient of the value function in Hybrid Dynamic Programming are governed by the same set of differential equations and have the same boundary conditions and hence are almost everywhere identical to each other along optimal trajectories. Analytic examples are provided to illustrate the results and, in particular, a Riccati formalism for linear quadratic hybrid tracking problems is presented.

\end{abstract}



\section{Introduction}

Pontryagin's Minimum Principle (MP) \cite{Pontryagin} and Bellman's Dynamic Programming (DP) \cite{Bellman} serve as the two key tools in optimal control theory. The Minimum Principle can be considered as a generalization of Hamilton's canonical system in Classical Mechanics as well as an extension of Weierstrass's necessary conditions in Calculus of Variations \cite{SussmannWillems}, while Dynamic Programming theory, including the Hamilton-Jacobi-Bellman (HJB) equation, may be considered as an extension of the Hamilton-Jacobi theory in Classical Mechanics and of Carathéodory's theorem in Calculus of Variations \cite{JacobsonMayne}. Both the MP and DP constitute necessary conditions for optimality which under certain assumptions become sufficient (see e.g. \cite{XYZ, Liberzon, FlemingRishel, Dreyfus, JacobsonMayne}). However, Dynamic Programming is widely used as a set of sufficient conditions for optimality after the optimal control extension of Carathéodory's sufficient conditions in Calculus of Variations (see e.g. \cite{JacobsonMayne, XYZ, Liberzon}).

The relationship between the Minimum Principle and Dynamic Programming, which were developed independently in 1950s, was addressed as early as the formal announcement of the Pontryagin Minimum Principle \cite{Pontryagin}. In the classical optimal control framework, this relationship has been elaborated by many others since then (see e.g. \cite{XYZ, Dreyfus, FlemingRishel, JacobsonMayne, ClarkeVinter, Vinter, Kim, CannarsaFrankowska, CerneaFrankowska, Zhou, SpeyerJacobson, Liberzon}). The result states that, under certain assumptions (see e.g. \cite{FlemingRishel, XYZ}), the adjoint process in the MP and the gradient of the value function in DP are equal, a property which we shall sometimes refer to as the adjoint-gradient relationship. While this relationship has been proved in various forms, the majority of arguments are based on the following two key elements: $\left(i\right)$ the assumption of the openness of the set of all points from which an optimal transition to the reference trajectory is possible \cite[p.~70]{Pontryagin} and $\left(ii\right)$ the inference of the extremality of the reference optimal state for the corresponding optimal control \cite[p.~72]{Pontryagin}. Then with the assumption of twice continuous differentiability of the value function, the method of characteristics (see e.g. \cite{FlemingRishel, XYZ}) can be employed to obtain the aforementioned relationship which is analogous to the derivation of the equivalence of the Hamiltonian system and the Hamilton-Jacobi equation. For certain classes of optimal control problems, the assumption of twice differentiability is intrinsically satisfied  since the total cost can become arbitrarily large and negative if the second partial derivative ceases to exist (see e.g. \cite{SpeyerJacobson}). But in general, even once differentiability of the value function is violated at certain points for numerous problems (see e.g. \cite{VinterBook, Zhou, Liberzon, Rozonoer, SpeyerJacobson, XYZ}). Consequently, the adjoint-gradient relationship is usually expressed within the general framework of nonsmooth analysis that declares the inclusion of the adjoint process in the set of generalized gradients of the value function \cite{ClarkeVinter, Vinter, Zhou, Kim, CannarsaFrankowska, CerneaFrankowska}. However, the general expression of the adjoint-gradient relationship in the framework of nonsmooth analysis is unnecessary for optimal control problems with appropriately smooth vector fields and costs, when the optimal feedback control possesses an admissible set of discontinuities \cite{FlemingRishel}. 


In contrast to classical optimal control theory, the relation between the Minimum Principle and Dynamic Programming in the hybrid systems framework has been the subject of limited number of studies (see e.g. \cite{APPECIFAC2014, APPECCDC2014, APPEC1ADHS2015}). One of the main difficulties in the discussion of hybrid systems
is that the term "hybrid" is not restrictive enough 
 and,
inevitably, the domains of definition of hybrid systems employed for the derivation of the results of hybrid optimal control theory (see e.g. \cite{ClarkeVinterMultiprocess, ClarkeVinterOptimal, SussmannNonSmooth, Sussmann, RiedingerKratzIungZanne, XuAntsaklis, ShaikhPEC, FarzinPECSIAM, FarzinPEC, GaravelloPiccoli, APPEC2016NAHS, APPEC2ADHS2015, DmitrukKaganovich, CapuzzoDolcettaEvans, BensoussanMenaldi, DharmattiRamaswamy, BarlesDharmattiRamaswamy, BranickyBorkerMitter, ZhangJames, PECEgerstedtMalhame, ScholligPECEgerstedt, ShaikhPECVerification, AzhmyakovBoltyanskiPoznyak, Lygeros, LygerosTomlinSastry, TomlinLygerosSastry, ZhuAntsaklis}) do not necessarily intersect in a general class of systems. This is especially due to the difference in approach and the assumptions required for the derivation of necessary and sufficient optimality conditions in the two key approaches. Hence, in order to establish the relationship between the Hybrid Minimum Principle (HMP) and Hybrid Dynamic Programming (HDP), there is the preliminary task of choosing an appropriately general class of hybrid systems within which the desired HMP and HDP results are valid. This is one of the tasks addressed in this paper before establishing the relation between the Minimum Principle and Dynamic Programming for classical and hybrid control systems.

The organisation of the paper is as follows: In Section \ref{sec:HybridSystems} a definition of hybrid systems is presented that covers a general class of nonlinear systems on Euclidean spaces with autonomous and controlled switchings and jumps allowed at the switching states and times. Further generalizations such as the lying of the system's vector fields in Riemannian spaces \cite{FarzinPECSIAM,FarzinPEC}, nonsmooth assumptions \cite{Sussmann, SussmannNonSmooth, ClarkeVinterOptimal, ClarkeVinterMultiprocess, BensoussanMenaldi, DharmattiRamaswamy}, and state-dependence of the control value sets \cite{GaravelloPiccoli}, as well as restrictions to certain subclasses such as those with regional dynamics \cite{PECEgerstedtMalhame, ScholligPECEgerstedt}, specified families of jumps \cite{BensoussanMenaldi, BranickyBorkerMitter, DharmattiRamaswamy, BarlesDharmattiRamaswamy} and systems with locally controllable trajectories \cite{ShaikhPEC} which are required for the sufficiency of the results are avoided so that for the selected class of hybrid systems, the associated Hybrid Minimum Principle, Hybrid Dynamic Programming theory, together with their relationships, are derived in a unified general framework. 
Similarly, the selected class of hybrid optimal control problems in Section \ref{sec:HOCP} covers a general class of hybrid optimal control problems with a large range of running, terminal and switching costs. With the exception of the infinite horizon problems considered in \cite{CapuzzoDolcettaEvans, BensoussanMenaldi, BranickyBorkerMitter, DharmattiRamaswamy, BarlesDharmattiRamaswamy}, this framework is in accordance with the majority of the work on the Hybrid Minimum Principle (HMP) (see \cite{SussmannNonSmooth, Sussmann, RiedingerKratzIungZanne, XuAntsaklis, ShaikhPEC, FarzinPECSIAM, FarzinPEC, GaravelloPiccoli, PassenbergLeiboldStursberg, APPECCDC2013, APPECCDC2014, APPECIFAC2014, APPEC1ADHS2015}) and a number of publications on Hybrid Dynamic Programming (HDP) (see e.g. \cite{PECEgerstedtMalhame, ScholligPECEgerstedt, ShaikhPECVerification, ZhangJames}) defined on finite horizons.

The statement of the Hybrid Minimum Principle is presented in Section \ref{sec:HMP} where necessary conditions are provided for the optimality of the trajectory and the controls of a hybrid system with fixed initial conditions and a sequence of autonomous or controlled switchings. These conditions are expressed in terms of the minimization of the distinct Hamiltonians indexed by the discrete state sequence of the hybrid trajectory. A feature of special interest is the set of boundary conditions on the adjoint processes and the Hamiltonian functions at autonomous and controlled switching times and states; these boundary conditions may be viewed as a generalization to the optimal control case of the Weierstrass–Erdmann conditions of the calculus of variations \cite{ShaikhPECErdmannnWeierstrass}. 

In Section \ref{sec:HDP}, the Principle of Optimality is employed to introduce the cost to go and the value functions. It is proved that on a bounded set in the state space, the cost to go functions are Lipschitz with a common Lipschitz constant which is independent of the control and hence their infimum, i.e. the value function, is Lipschitz with the same Lipschitz constant. The necessary conditions of Hybrid Dynamic Programming are then established in the form of the Hamilton-Jacobi-Bellman (HJB) equation and the corresponding boundary conditions. 

In Section \ref{sec:HMPvsHDP}, the main result is given describing the relationship of the HMP and HDP. The proof is different in approach from the classical arguments discussed earlier and in particular, the sequence of proof steps appear in a different order. To be specific, the optimality condition, i.e. the Hamiltonian minimization property $\left(ii\right)$ discussed earlier, appears in the last step in order to emphasise the independence of the dynamics of the cost gradient process from the optimality of the control input. Consequently, assumption $\left(i\right)$ is used differently here from the classical proof methods and in particular, the optimality of the transitions back to the reference trajectory is relaxed to the existence of (not necessarily optimal) neighbouring trajectories. After the derivation of the dynamics and boundary conditions for the cost gradient, or sensitivity, corresponding to an arbitrary control input, it is shown that an optimal control leads to the same dynamics and boundary conditions for the optimal cost gradient process as those for the adjoint process. Thus by the existence and uniqueness properties of the governing ODE solutions, it is concluded that the optimal cost gradient, i.e. the gradient of the value function generated by the HJB, is equal to the adjoint process in the corresponding HMP formulation. 

Section \ref{sec:Examples} provides illustrative examples and in Section \ref{sec:Riccati} a Riccati formalism for optimal tracking problems associated with linear quadratic hybrid systems is presented.



\section{Hybrid Systems}
\label{sec:HybridSystems}
\begin{definition}
A \textit{hybrid system (structure)} $\mathbb{H}$ is a septuple
\begin{equation}
\text{\ensuremath{\mathbb{H}}}=\left\{ H,I,\Gamma,A,F,\Xi,\mathcal{M}\right\}, \label{Hybrid System}
\end{equation}
where the symbols in the expression and their governing assumptions are defined as below.

\textbf{A0:} $H:=Q\times M$ is called the \textit{(hybrid) state space}
of the hybrid system $\text{\ensuremath{\mathbb{H}}}$, where

$Q=\left\{ 1,2,...,\left|Q\right|\right\} \equiv\left\{ q_{1},q_{2},...,q_{\left|Q\right|}\right\} ,\left|Q\right|<\infty$, is a finite set of \textit{discrete states (components)}, and

$M=\left\{\mathbb{R}^{n_{q}}\right\}_{q\in Q}$ is a family of finite dimensional continuous valued state spaces, where $n_{q}\leq n<\infty$ for all $q\in Q$.

$I:=\Sigma\times U$ is the set of system input values, where

$\Sigma$ with $\left|\Sigma\right|<\infty$ is the set of discrete state transition and \text{continuous} state jump events extended with the identity \text{element}, and

$U=\left\{ U_{q}\right\}_{q\in Q}$ is the set of \textit{admissible input control values}, where each $U_q \subset\text{\ensuremath{\mathbb{R}}}^{m_q}$ is a compact set in $\text{\ensuremath{\mathbb{R}}}^{m_q}$.

The set of admissible (continuous) control inputs $\mathcal{U}\left(U\right):=L_{\infty}\left(\left[t_{0},T_{*}\right),U\right)$, is defined to be the set of all measurable functions that are bounded up to a set of measure zero on $\left[t_{0},T_{*}\right),T_{*}<\infty$. The boundedness property necessarily holds since admissible input functions take values in the compact set $U$.

$\Gamma:H\times\Sigma\rightarrow H$ is a time independent (partially defined) \textit{discrete state transition map}.

$\Xi:H\times\Sigma\rightarrow H$ is a time independent (partially defined) \textit{continuous state jump transition map}. All $\xi_{\sigma}\in\Xi$, $\xi_{\sigma} : \mathbb{R}^{n_q} \rightarrow  \mathbb{R}^{n_p}$, $p\in A\left(q,\sigma\right)$ are assumed to be continuously differentiable in the
continuous state $x \in \mathbb{R}^{n_q}$. 

$A:Q\times\Sigma\rightarrow Q$ denotes both a deterministic finite automaton and the automaton's associated transition function on the state space $Q$ and event set $\Sigma$, such that for a discrete state $q\in Q$ only the discrete controlled and uncontrolled transitions into the $q$-dependent subset $\left\{ A\left(q,\sigma\right),\sigma\in\Sigma\right\} \subset Q$ occur under the projection of $\Gamma$ on its $Q$ components: $\Gamma:Q\times\mathbb{R}^{n}\times\Sigma\rightarrow H|_{Q}$. In other words, $\Gamma$ can only make a discrete state transition in a hybrid state $\left(q,x\right)$ if the automaton $A$ can make the corresponding transition in $q$.

$F$ is an indexed collection of \textit{vector fields} $\left\{ f_{q}\right\}_{q\in Q}$ such that $f_{q}\in C^{k_{f_q}}\left(\mathbb{R}^{n_q}\times U_q\rightarrow\mathbb{R}^{n_q}\right)$, $k_{f_q}\geq1$, satisfies a joint uniform Lipschitz condition, i.e. there exists $L_{f}<\infty$ such that $\left\Vert f_{q}\left(x_{1},u_{1}\right)-f_{q}\left(x_{2},u_{2}\right)\right\Vert \leq L_{f}\left(\left\Vert x_{1}-x_{2}\right\Vert +\left\Vert u_{1}-u_{2}\right\Vert \right)$, for all $x, x_{1},x_{2}\in\mathbb{R}^{n_q}$, $u, u_1, u_2\in U_q$, $q\in Q$. 

$\mathcal{M}=\left\{ m_{\alpha}:\alpha\in Q\times Q,\right\} $ denotes a collection of \textit{switching manifolds} such that, for any ordered pair $\alpha \equiv \left(\alpha_1,\alpha_2\right) =\left(p,q\right)$, $m_{\alpha}$ is a smooth, i.e. $C^{\infty}$, codimension $1$ sub-manifold of $\mathbb{R}^{n_{\alpha_1}}$, described locally by $m_{\alpha}=\left\{ x:m_{\alpha}\left(x\right)=0\right\}$. It is assumed that $m_{\alpha}\cap m_{\beta}=\emptyset$, whenever $\alpha_1 = \beta_1$ but $\alpha_2 \neq \beta_2$, for all $\alpha,\beta\in Q\times Q$.
\hfill $\square$
\end{definition}
We note that the case where $m_{\alpha}$ is identified with its reverse ordered version $m_{\bar{\alpha}}$ giving $m_{\alpha}=m_{\bar{\alpha}}$ is not ruled out by this definition, even in the non-trivial case $m_{p,p}$ where $\alpha_1 = \alpha_2 = p$. The former case corresponds to the common situation where the switching of vector fields at the passage of the continuous trajectory in one direction through a switching manifold is reversed if a reverse passage is performed by the continuous trajectory, while the latter case corresponds to the standard example of the bouncing ball. 

Switching manifolds will function in such a way that whenever a trajectory governed by the controlled vector field meets the switching manifold transversally there is an autonomous switching to another controlled vector field or there is a jump transition in the continuous state component, or both. A transversal arrival on a switching manifold $m_{q,r}$, at state $x_q \in m_{q,r}=\left\{ x \in \mathbb{R}^{n_{q}} : m_{q,r}\left(x\right)=0 \right\}$ occurs whenever
\begin{equation}
{\nabla m_{q,r} \left(x_q\right)}^T f_q \left(x_q,u_q\right) \neq 0 ,
\label{TransversalityOfTrajectoriesToManifolds}
\end{equation}
for $u_q\in U_q$, and $q,r \in Q$. It is further assumed that:


\textbf{A1:} The initial state $h_{0}:=\left(q_{0},x\left(t_{0}\right)\right)\in H$ is such that $m_{q_{0},q_{j}}\left(x_{0}\right)\neq0$, for all $q_{j}\in Q$.  
\hfill $\square$

\begin{definition}
\label{def:HybridInputProcess}
A \textit{hybrid input process} is a pair $I_{L}\equiv I_{L}^{\left[t_{0},t_{f}\right)}:= \left(S_L,u\right)$ defined on a half open interval $\left[t_{0},t_{f}\right)$, $t_{f}<\infty$, where $u\in{\cal U}$ and $S_L = \left(\left(t_{0},\sigma_{0}\right), \left(t_{1},\sigma_{1}\right), \cdots, \left(t_{L},\sigma_{L}\right)\right)$, $L<\infty$, is a finite \textit{hybrid sequence of switching events} consisting of a strictly increasing sequence of times $\tau_L := \left\{ t_{0},t_{1},t_{2},\ldots,t_{L}\right\}$ and a \textit{discrete event sequence} $\sigma$ with $\sigma_0 = id$ and $\sigma_i \in \Sigma$, $i \in \left\{1,2,\cdots,L\right\}$.\hfill $\square$
\end{definition}

\begin{definition}
\label{def:HybridStateProcess}
A \textit{hybrid state process} (or \textit{trajectory}) is a triple $\left(\tau_L,q,x\right)$ consisting of the sequence of switching times $\tau_L = \left\{ t_{0},t_{1},\ldots,t_{L}\right\}$, $L<\infty$, the associated sequence of discrete states $q=\left\{ q_{0},q_{1},\ldots,q_{L}\right\}$, and the sequence ${x\left(\cdot\right)=\left\{ x_{q_{0}}\left(\cdot\right),x_{q_{1}}\left(\cdot\right),\ldots,x_{q_{L}}\left(\cdot\right)\right\}}$ of piece-wise differentiable functions $x_{q_{i}}\left(\cdot\right):\left[t_{i},t_{i+1}\right)\rightarrow\mathbb{R}^{n_{q_i}}$.\hfill $\square$
\end{definition}

\begin{definition}
The \textit{input-state trajectory} for the hybrid system $\mathbb{H}$ satisfying A0 and A1 is a hybrid input $I_L = \left(S_L,u\right)$ together with its corresponding hybrid state trajectory $\left(\tau_L,q,x\right)$ defined over $\left[t_{0},t_{f}\right),t_{f}<\infty$, such that it satisfies:

$\left(i\right)$ \textit{Continuous State Dynamics:} The continuous state \text{component} $x\left(\cdot\right)=\left\{ x_{q_{0}}\left(\cdot\right),x_{q_{1}}\left(\cdot\right),\ldots,x_{q_{L}}\left(\cdot\right)\right\}$ is a piecewise continuous function which is almost everywhere differentiable and on each time segment specified by $\tau_L$ satisfies the dynamics equation
\begin{equation}
{\dot{x}_{q_{i}}\left(t\right) = f_{q_{i}}\left(x_{q_{i}}\left(t\right),u\left(t\right)\right)}, \hspace{1 cm} {a.e.\; t\in\left[t_{i},t_{i+1}\right)},
\end{equation}
with the initial conditions 
\begin{align}
x_{q_{0}}\left(t_{0}\right) &=x_{0} ,
\\[-\belowdisplayskip]
\vspace{-7mm} x_{q_{i}}\left(t_{i}\right) &= \xi_{\sigma_{i}}\left(x_{q_{i-1}}\left(t_{i}-\right)\right) := \xi_{\sigma_{i}}\left(\lim_{t\uparrow t_{i}}x_{q_{i-1}}\left(t\right)\right) , 
\end{align}
for $\left(t_{i},\sigma_{i}\right) \in S_L$. In other words, $x\left(\cdot\right)=\left\{ x_{q_{0}}\left(\cdot\right),x_{q_{1}}\left(\cdot\right),\ldots,x_{q_{L}}\left(\cdot\right)\right\}$ is a piecewise continuous function which is almost everywhere differentiable and is such that each $x_{q_{i}}\left(\cdot\right)$ satisfies 
\begin{equation}
x_{q_{i}}\left(t\right)=x_{q_{i}}\left(t_{i}\right)+\int_{t_{i}}^{t}f_{q_{i}}\left(x_{q_{i}}\left(s\right),u\left(s\right)\right)ds ,
\end{equation}
for $t \in \left[t_{i},t_{i+1}\right)$.

$\left(ii\right)$  \textit{Autonomous Discrete Transition Dynamics:} An \text{autonomous} (uncontrolled) discrete state transition from $q_{i-1}$ to $q_i$ together with a continuous state jump $\xi_{\sigma_i}$ occurs at the \textit{autonomous switching time} $t_i$ if $x_{q_{i-1}}\left(t_{i}-\right):=\lim_{t\uparrow t_{i}}x_{q_{i-1}}\left(t\right)$ satisfies a switching manifold condition of the form
\begin{equation}
m_{q_{i-1}q_i}\left(x_{q_{i-1}}\left(t_{i}-\right)\right) = 0 ,
\end{equation}
for $q_i\in Q$, where $m_{q_{i-1}q_i}\left(x\right) = 0$ defines a $\left(q_{i-1},q_i\right)$ switching manifold and it is not the case that either $\left(i\right)$ $x_{q_{i-1}}\left(t_{i}-\right)\in \partial m_{q_{i-1}q_i}$ or $\left(ii\right)$ $f_{q_i-1}\left(x_{q_{i-1}}\left(t_{i}-\right),u_{q_{i-1}}\left(t_{i}-\right)\right) \perp \nabla m_{q_{i-1}q_i}\left(x_{q_{i-1}}\left(t_{i}-\right)\right)$, i.e. $t_i$ is not a manifold termination instant (see \cite{PECHybridNotes}). With the assumptions A0 and A1 in force, such a transition is well defined and labels the event $\sigma_{q_{i-1}q_i} \in \Sigma$, that corresponds to the hybrid state transition
\begin{equation}
h\left(t_{i}\right) \! \equiv \! \left(q_{i},x_{q_{i}}\left(t_{i}\right)\right) \!=\! \left(\Gamma \! \left(q_{i-1} \!,\sigma_{q_{i-1}q_i}\right),\xi_{\sigma_{q_{i-1}q_i}} \!\! \left(x_{q_{i-1}}^{\left(t_{i}-\right)}\right)\right) .
\end{equation}

$\left(iii\right)$  \textit{Controlled Discrete Transition Dynamics:} A controlled discrete state transition together with a controlled continuous state jump $\xi_{\sigma}$ occurs at the \textit{controlled discrete event time} $t_i$ if $t_i$ is not an autonomous discrete event time and if there exists a controlled discrete input event $\sigma_{q_{i-1}q_i} \in \Sigma$ for which
\begin{equation}
h\left(t_{i}\right) \! \equiv \! \left(q_{i},x_{q_{i}}\left(t_{i}\right)\right) \!=\! \left(\Gamma \! \left(q_{i-1} \!,\sigma_{q_{i-1}q_i}\right),\xi_{\sigma_{q_{i-1}q_i}} \!\! \left(x_{q_{i-1}}^{\left(t_{i}-\right)}\right)\right),
\end{equation}
with $\left(t_{i},\sigma_{q_{i-1}q_{i}}\right) \in S_L$ and $q_{i}\in A\left(q_{i-1}\right)$.
\hfill $\square$
\end{definition}
\vspace{6pt}

\begin{theorem} \textnormal{\textbf{Existence and Uniqueness of Solution Trajectories for Hybrid Systems \cite{PECHybridNotes}:}}
\label{theorem:ExistenceUniqueness}
A hybrid system $\mathbb{H}$ with an initial hybrid state $\left(q_{0},x_{0}\right)$ satisfying assumptions A0 and A1 possesses a unique hybrid input-state trajectory on $\left[t_{0},T_{**}\right)$, where $T_{**}$ is the least of

$\left(i\right)$ $T_{*} \leq \infty $, where $\left[t_{0},T_{*}\right)$ is the temporal domain of the definition of the hybrid system,

$\left(ii\right)$ a manifold termination instant $T_{*}$ of the trajectory $h\left(t\right) = h\left(t,\left(q_0,x_0\right),\left(S_L,u\right) \right)$, $t\geq t_0$, at which either $x\left(T_{*}-\right) \in \partial m_{q\left(T_*-\right)q\left(T_*\right)}$ or $f_{q\left(T_*-\right)}\left(x\left(T_*-\right),u\left(T_*-\right)\right) \perp \nabla m_{q\left(T_*-\right)q\left(T_*\right)}\left(x\left(T_{*}-\right)\right)$.
\hfill $\square$
\end{theorem} 
We note that Zeno times, i.e. accumulation points of discrete transition times, are ruled out by Definitions 2, 3 and 4. 
\vspace{6pt}



\section{Hybrid Optimal Control Problem}
\label{sec:HOCP}
\textbf{A2:} Let $\left\{ l_{q}\right\}_{q\in Q},l_{q}\in C^{n_{l}}\left(\mathbb{R}^{n_q}\times U_q\rightarrow\mathbb{R}_{+}\right),n_{l} \geq 1$, be a family of cost functions with $n_l = 2$ unless otherwise stated; $\left\{ c_{\sigma_j}\right\}_{\sigma_j\in\Sigma}\in C^{n_{c}}\left(\mathbb{R}^{n_{q_{j-1}}}\times\Sigma\rightarrow\mathbb{R}_{+}\right),n_{c}\geq1$, be a family of switching cost functions; and $g\in C^{n_{g}}\left(\mathbb{R}^{n_{q_f}}\rightarrow\mathbb{R}_{+}\right),n_{g}\geq1$,
be a terminal cost function satisfying the following assumptions:

$\left(i\right)$ There exists $K_{l}<\infty$ and $1\leq\gamma_{l}<\infty$ such that
$\left|l_{q}\left(x,u\right)\right|\leq K_{l}\left(1+\left\Vert x\right\Vert ^{\gamma_{l}}\right)$ and $\left|l_{q}\left(x_{1},u_{1}\right)-l_{q}\left(x_{2},u_{2}\right)\right|\leq K_{l}\left(\left\Vert x_{1}-x_{2}\right\Vert +\left\Vert u_{1}-u_{2}\right\Vert \right)$, for all
$x\in\mathbb{R}^{n_q}$, $u\in U_q$, $q\in Q$.

$\left(ii\right)$ There exists $K_{c}<\infty$ and $1\leq\gamma_{c}<\infty$ such that
$\left|c_{\sigma_j}\left(x\right)\right|\leq K_{c}\left(1+\left\Vert x\right\Vert ^{\gamma_{c}}\right)$,
$x\in\mathbb{R}^{n_{q_{j-1}}},\sigma_j\in\Sigma$.

$\left(iii\right)$ There exists $K_{g}<\infty$ and $1\leq\gamma_{g}<\infty$ such that
$\left|g\left(x\right)\right|\leq K_{g}\left(1+\left\Vert x\right\Vert ^{\gamma_{g}}\right)$,
$x\in\mathbb{R}^{n_{q_{f}}}$.

\hfill $\square$

Consider the initial time $t_{0}$, final time $t_{f}<\infty$, and initial
hybrid state $h_{0}=\left(q_{0},x_{0}\right)$. With the number
of switchings $L$ held fixed, the set of all hybrid input trajectories  in Definition \ref{def:HybridInputProcess} with exactly $L$ switchings is denoted by $\bm{I_{L}}$, and for all $I_{L} :=\left(S_{L},u\right) \in \bm{I_{L}}$ the hybrid switching sequences take the form
$S_{L}= \left\{ \left(t_{0},id\right),\left(t_{1},\sigma_{q_{0}q_{1}}\right),  \ldots, \left(t_{L},\sigma_{q_{L-1}q_{L}}\right)\right\}
\equiv\left\{ \left(t_{0},q_{0}\right),\left(t_{1},q_{1}\right),\ldots,\left(t_{L},q_{L}\right)\right\}$ and the corresponding continuous control inputs are of the form $u\in\mathcal{U} = \bigcup_{i=0}^{L} L_{\infty}\left(\left[t_i,t_{i+1}\right),U_{q_i}\right)$, where $t_{L+1}=t_f$.

Let $I_{L}$ be a hybrid input trajectory that by Theorem \ref{theorem:ExistenceUniqueness} results in a unique hybrid state process. Then hybrid performance functions for the corresponding hybrid input-state trajectory are defined as
\begin{multline}
J\left(t_{0},t_{f},h_{0},L;I_{L}\right):=
\sum_{i=0}^{L}\int_{t_{i}}^{t_{i+1}}l_{q_{i}}\left(x_{q_{i}}\left(s\right),u_{q_{i}}\left(s\right)\right)ds
\\
+\sum_{j=1}^{L}c_{\sigma_{j}}\left(t_{j},x_{q_{j-1}}\left(t_{j}-\right)\right)+g\left(x_{q_{L}}\left(t_{f}\right)\right) .
\label{Hybrid Cost}
\end{multline}

\begin{definition}
\label{def:BHOCP}
The \textit{Bolza Hybrid Optimal Control Problem} (BHOCP) is defined as the infimization of the hybrid cost \eqref{Hybrid Cost} over the family of hybrid input trajectories $\bm{I_{L}}$, i.e.
\begin{equation}
J^{o}\left(t_{0},t_{f},h_{0},L\right)=\inf_{I_{L} \in \bm{I_{L}}}J\left(t_{0},t_{f},h_{0},L;I_{L}\right) . \label{HOCP}
\end{equation}
\hfill $\square$
\end{definition}

\begin{definition}
\label{def:MHOCP}
The \textit{Mayer Hybrid Optimal Control Problem} (MHOCP) is defined as a special case of the BHOCP where $l_q\left(x_q,u_q\right) = 0$ for all $q \in Q$ and $c_{\sigma_j}\left(t_j,x_{q_{j-1}}\left(t_{j}-\right)\right) = 0$ for all $\sigma_j \in \Sigma$, $1 \leq j \leq L$.
\hfill $\square$
\end{definition}

\begin{remark}
\label{rem:BHOCPandMHOCPrelationship}
\textit{The Relationship between Bolza and Mayer Hybrid Optimal Control Problems:} 
In general, a BHCOP can be converted into an MHCOP with the introduction of the auxiliary state component $z$ and the extension of the continuous valued state to
\begin{equation}
\hat{x}_{q}:=\left[\begin{array}{c}
z_{q}\\
x_{q}
\end{array}\right] .\label{ExtendedState}
\end{equation}

With the definition of the augmented vector fields as
\begin{equation}
\dot{\hat{x}}_{q}=\hat{f}_{q}\left(\hat{x_q},u_q\right):=\left[\begin{array}{c}
l_{q}\left(x_q,u_q\right)\\
f_{q}\left(x_q,u_q\right)
\end{array}\right],\label{ExtendedField}
\end{equation}
subject to the initial condition
\begin{equation}
\hat{h}_{0}=\left(q_{0},\hat{x}_{q_{0}}\left(t_{0}\right)\right)=\left(q_{0},\left[\begin{array}{c}
0\\
x_{0}
\end{array}\right]\right),\label{ExtendedIC}
\end{equation}
and with the switching boundary conditions governed by the extended jump
function defined as
\begin{equation}
\hat{x}_{q_j} \! \left(t_{j}\right)=\hat{\xi}_{\sigma_j}\left(\hat{x}_{q_{j-1}}\! \left(t_{j}-\right)\right):=\!\! \left[\!\!\!\! \begin{array}{c}
z_{q_{j-1}} \! \left(t_{j}-\right)+c_{\sigma_j}\left(x_{q_{j-1}} \! \left(t_{j}-\right)\right)\\
\xi_{\sigma_j}\left(x_{q_{j-1}}\left(t_{j}-\right)\right)
\end{array} \!\!\!\!\! \right],\label{ExtendedJump2}
\end{equation}
the cost  \eqref{Hybrid Cost} of the BHOCP turns into the Mayer form with
\begin{equation}
J\left(t_{0},t_{f},\hat{h}_{0},L;I_{L}\right):=\hat{g}\left(\hat{x}_{q_{L}}\left(t_{f}\right)\right),\label{Mayer - full transformation-1}
\end{equation}
where $\hat{g}\left(\hat{x}_{q_{L}}\left(t_{f}\right)\right)=z_{q_L}\left(t_{f}\right)+g\left(x_{q_L}\left(t_{f}\right)\right)$.
\hfill $\square$
\end{remark}



\section{Hybrid Minimum Principle}
\label{sec:HMP}

\begin{theorem} \textnormal{\textbf{\cite{PhDthesisAP}}} 
\label{theorem:HMP}
Consider the hybrid system $\mathbb{H}$ subject to assumptions A0-A2, and the HOCP \eqref{HOCP} for the hybrid performance function \eqref{Hybrid Cost}.
Define the family of system Hamiltonians by 
\begin{equation}
H_{q}\left(x_q,\lambda_q,u_q\right)=\lambda_q^{T}f_{q}\left(x_q,u_q\right)+l_{q}\left(x_q,u_q\right),\label{Hamiltonian - Bolza}
\end{equation}
$x_q, \lambda_q \in \mathbb{R}^{n_{q}}$, $u_q \in U_{q}$, $q\in Q$. Then for an optimal switching sequence $q^{o}$ and along the corresponding optimal trajectory $x^{o}$, there exists an adjoint process $\lambda^{o}$ such that
\begin{align}
\dot{x}_q^{o} &=\frac{\partial H_{q^{o}}}{\partial \lambda_q}\left(x_q^{o},\lambda_q^{o},u_q^{o}\right), \label{StateDynamics}
\\
\dot{\lambda}_q^{o} &=-\frac{\partial H_{q^{o}}}{\partial x_q}\left(x_q^{o},\lambda_q^{o},u_q^{o}\right),\label{lambda dynamics}
\end{align}
almost everywhere $\; t\in\left[t_{0},t_{f}\right]$ with
\begin{align}
x^o_{q^o_0}\left(t_0\right) &=x_0,\\
x^o_{q^o_j}\left(t_{j}\right) &=\xi_{\sigma_j}\left(x^o_{q^o_{j-1}}\left(t_{j}-\right)\right), \label{StateBC} \\
\lambda^{o}_{q^o_L}\left(t_{f}\right) &=\nabla g\left(x^{o}_{q^o_L}\left(t_{f}\right)\right),\label{lambda final condition}\\
\hspace{-2mm} \lambda^{o}_{q^o_{j-1}}\left(t_{j}-\right)\equiv\lambda^{o}_{q^o_{j-1}}\left(t_{j}\right) &={\nabla \xi_{\sigma_j}}^{T}\lambda^{o}_{q^o_j}\left(t_{j}+\right) + \nabla c_{\sigma_j}  +p \nabla m,  \label{lambda boundary condition}
\end{align}
where $p\in\mathbb{R}$ when $t_{j}$ indicates the time of an autonomous switching, and $p=0$ when $t_{j}$ indicates the time of a controlled switching. Moreover, 
\begin{equation}
H_{q^{o}}\left(x^{o},\lambda^{o},u^{o}\right)\leq H_{q^{o}}\left(x^{o},\lambda^{o},u\right),\label{HminWRTu}
\end{equation}
for all $u\in U_{q^o}$, that is to say the Hamiltonian is minimized with respect to the control input, and at a switching time $t_{j}$ the Hamiltonian
satisfies
\begin{multline}
H_{q_{j-1}^{o}}\left.\left(x^{o},\lambda^{o},u^{o}\right)\right|_{t_{j}-}\equiv H_{q_{j-1}^{o}}\left(t_{j}\right)
\\
=H_{q_{j}^{o}}\left(t_{j}\right)\equiv H_{q_{j}^{o}}\left.\left(x^{o},\lambda^{o},u^{o}\right)\right|_{t_{j}+}.\label{Hamiltonian jump}
\end{multline}
\hfill $\square$
\end{theorem}

We note that the gradient of the state transition jump map $\nabla \xi_{\sigma_j} = \frac{\partial \xi_{\sigma_j}{(x_{q_{j-1}})} }{\partial x_{q_{j-1}}} \equiv \frac{\partial x_{q_j}}{\partial x_{q_{j-1}}} $ is not necessarily square due to the possibility of changes in the state dimension, but the boundary conditions \eqref{lambda boundary condition} are well-defined for hybrid optimal control problems satisfying A0-A2.

\begin{remark}
\label{remark:BolzaMayer}
For an HOCP represented in the Bolza form (as in Definition \ref{def:BHOCP}) and the corresponding Mayer representation (as in Definition \ref{def:MHOCP} and through Remark \ref{rem:BHOCPandMHOCPrelationship}) the corresponding HMP results are given by (see also \cite{APPECCDC2014})
\begin{multline}
\hat{\lambda}_q^T \hat{f}_q\left(\hat{x}_q,u_q\right)\equiv\hat{H}_q\left(\hat{x},\hat{\lambda},u\right) 
\\
=H_q\left(x,\lambda,u\right) \equiv \lambda_q^T f_q\left(x_q,u_q\right) + l_q\left(x_q,u_q\right) ,
\label{HamiltonianEquivalence}
\end{multline}
and 
\begin{equation}
\hat{\lambda}_q\left(t\right)=\left[\begin{array}{c}
1\\
\lambda_q\left(t\right)
\end{array}\right] . \label{ExtendedAdjoint}
\end{equation}
\hfill $\square$
\end{remark}



\section{Hybrid Dynamic Programming}
\label{sec:HDP}

\begin{definition}
Consider the hybrid system \eqref{Hybrid System} and the class of HOCP \eqref{HOCP} with the hybrid cost \eqref{Hybrid Cost}. At an arbitrary instant $t\in\left[t_{0},t_{f}\right]$, an initial hybrid state $h\!=\!\left(q,\! x\right) \!\in\! Q \!\times\! \mathbb{R}^{n_q}$, and a specified number of remaining switchings $L\!-\!j\!+\!1$ corresponding to $t \!\in\! \left(t_{j\!-\!1},t_{j}\right]$, the \textit{cost to go} subject to the hybrid input process $I_{L-j+1} \equiv I_{L-j+1}^{\left[t,t_{f}\right]}$ is given by \vspace{-2mm}
\begin{multline}
J\left(t,t_{f},q,x,L-j+1;I_{L-j+1}\right) \vspace{-6pt}
\\[-4pt]
=\int_{t}^{t_{j}}l_{q}\left(x_q,u_q\right)ds+\sum_{i=j}^{L}c_{\sigma_{q_{i-1}q_{i}}}\left(t_{i},x_{q_{i-1}}\left(t_{i}-\right)\right)
\\[-\belowdisplayskip]
+\sum_{i=j}^{L}\int_{t_{i}}^{t_{i+1}}l_{q_{i}}\left(x_{q_{i}}\left(s\right),u\left(s\right)\right)ds+g\left(x_{q_{L}}\left(t_{f}\right)\right).
\label{CostToGoDefinition}
\end{multline}

The \textit{value function} is defined as the optimal cost to go over the corresponding family of hybrid control inputs, i.e. \vspace{-4pt}
\begin{equation}
V\left(t,q,x,L\!-\!j\!+\!1\right):=\inf_{I_{L\!-\!j\!+\!1}}J\left(t,t_{f},q,x,L\!-\!j\!+\!1;I_{L\!-\!j\!+\!1}\right) . \hspace{-4pt}
\label{ValueFunctionDefinition}
\vspace{-10pt}
\end{equation}
\hfill $\square$
\end{definition}

\begin{theorem}
\label{theorem:VisLipschitz}
With the assumptions A0-A2 in force, the value function \eqref{ValueFunctionDefinition} is Lipschitz in $x$ uniformly in $t$ for all $t\in \bigcup_{i=0}^{L} \left(t_i,t_{i+1}\right)$, i.e. for $B_{r}=\left\{ x\in\mathbb{R}^{n}:\left\Vert x\right\Vert <r\right\}$ and for all $t\in \left(t_i,t_{i+1}\right)$, $x \equiv x_t \in B_r$ there exist a neighbourhood $N_{r_x}\left(x_t\right)$ and a constant $0<K<\infty$ such that
\begin{equation}
\!\left|\! V\!\left(t,q,x_{t}\!,\! L\!-\! j\!+\!1\right)\!-\! V\!\left(s,q,x_{s}\!,\! L\!-\! j\!+\!1\right)\!\right|\!\!<\!\! K\!\left(\!\left\Vert \! x_{t}\!-\! x_{s}\!\right\Vert ^{2}\!\!+\!\left|\! s\!-\! t\!\right|^{2}\!\right)\!\!^{\frac{1}{2}},
\label{VisLipschitz}
\end{equation}
for $s\in \left(t_i,t_{i+1}\right)$ and $x_s \in N_{r_x}\left(x_t\right)$.
\hfill $\square$
\end{theorem}

\begin{proof}
See Appendix \ref{sec:AppendixVisLipschitz}.
\end{proof}

\begin{definition}
Let $M_{\left(i\right)}$ denote the set of all ${\left(t,x\right)}\in {\mathbb{R} \times \bigcup_{i=0}^{L} \mathbb{R}^{n_{q_i}}}$ for which the $i$'th derivatives of $V$ exist and are continuous.
\hfill $\square$
\end{definition}

Note that from Theorem \ref{theorem:VisLipschitz}, it is concluded that $M_{\left(0\right)} \supseteq \bigcup_{i=0}^{L} \left(t_i,t_{i+1}\right) \times \mathbb{R}^{n_{q_i}}$, i.e. the value function is at most discontinuous at the switching instants with non-zero switching costs and non-identity jump maps.

\begin{corollary}
\label{corollary:VisDifferentiable}
From Theorem \ref{theorem:VisLipschitz} and Rademacher's theorem (see e.g. \cite{FlemingRishel, JafarpourLewis, Federer}), the Lipschitz property of the value function implies the differentiability almost everywhere, and hence the set $M_{\left(1\right)}$ is dense in $\bigcup_{i=0}^{L} \left[t_i,t_{i+1}\right] \times \mathbb{R}^{n_{q_i}}$.
\hfill $\square$
\end{corollary}


\begin{theorem} \textnormal{\textbf{\cite{PhDthesisAP}}}
\label{theorem:HDP}
Consider the hybrid system $\mathbb{H}$ and the HOCP \eqref{HOCP} together with the assumptions A0-A2 as above.
Then for all $\left(t,x_{q_i}\right)\in M_{\left(1\right)}$, $q_i\in Q$, the  Hamilton-Jacobi-Bellman (HJB) equation holds, i.e.
\begin{equation}
-\frac{\partial V}{\partial t}=\inf_{u_{q_{i}}}\left\{ l_{q_{i}}\left(x_{q_{i}},u_{q_{i}}\right)+\left\langle \nabla V,f_{q_{i}}\left(x_{q_{i}},u_{q_{i}}\right)\right\rangle \right\} ,\label{HJB}
\end{equation}
$a.e.\; t\in\left(t_{i},t_{i+1}\right)$, $0 \leq i \leq L$ subject to the terminal condition
\begin{equation}
V\left(t_{f},q_{L},x,0\right)=g\left(x_{q_L}\right),\label{V terminal}
\end{equation}
and at the switching times $t_j \in \tau_L = \left\{t_1, \cdots, t_L\right\}$ subject to the boundary conditions
\begin{equation}
V\!\left(t_{j},q,x,L\!-\! j\!+\!1\!\right)\!=\!\!\min_{\!\sigma_{j}\in\Sigma_{j}\!}\!\!\left\{ \! V\!\left( t_{j},\Gamma\!\left(q\!,\!\sigma_{\! j}\right)\!,\xi_{\!\sigma_{j}}\!\left(x\right)\!,L\!-\! j\right)\!+\! c_{\!\sigma_{j}}\!\left(x\right)\!\right\} \!,
\label{VminWRTtjSigma}
\end{equation}
and
\vspace{-3mm}
\begin{multline}
\hspace{-7pt} l_{q}\left(x,u^{o}\left(t_{j}-,x\right)\right)+\left\langle \nabla V,f_{q}\left(x,u^o\left(t_{j}-,x\right)\right)\right\rangle 
\\
\hspace{-5pt}\equiv\!\frac{\!-\partial}{\partial t}V\left(t_{j}\!-,q,x,L\!-\! j\!+\!1\right)\!=\!\frac{\!-\partial}{\partial t}V\left(t_{j},\Gamma\left(q,\sigma_{j}\right),\xi_{\sigma_{j}}\left(x\right)\!,L\!-\! j\right)
\\
\equiv l_{_{\Gamma\left(q,\sigma_{j}\right)}}\!\!\!\left(\xi_{\sigma_{j}}^{\left(x\right)}\!,u^{o}\!\!\left(t_{j},\xi_{\sigma_{j}}^{\left(x\right)}\!\right)\right)\!+\!\left\langle \!\nabla V,f_{_{\Gamma\left(q,\sigma_{j}\right)}}\!\!\!\left(\xi_{\sigma_{j}}^{\left(x\right)}\!,u^{o}\!\!\left(t_{j},\xi_{\sigma_{j}}^{\left(x\right)}\!\right)\!\right)\!\!\right\rangle , \hspace{-10pt}
\label{HamiltonianContinuityForV}
\end{multline}
where if $\left(t_{j},x\right)\in\mathbb{R}\times\mathbb{R}^{n_{q_{j-1}}}$ belong to a controlled switching set then $\Sigma_{j} = \Sigma$ subject to the automaton constraint
that $\Gamma\left(q, \sigma_j\right)$ is defined; and in the case of an autonomous switching, the set $\Sigma_{j}$ is reduced to a subset of discrete inputs which are consistent with the switching manifold condition $m_{q,{\Gamma\left(q, \sigma_j\right)}}\left(x_q\right)=0$. In the above equation, the notation $u^o\left(t,x\right)$ indicates the optimal [continuous valued] input corresponding to $x$ at the instant $t$.
\hfill $\square$
\end{theorem}


\begin{definition}
\label{def:InputWithAdmissibleSetOfDisc}
A feedback control input $I_{L-j+1}^{\left[t,t_{f}\right]} \left(t,q,x\right)=\left(S_{L-j+1},u_{q\left(\tau\right)}\left(\tau,x\right)\right)$, $\tau \in \left[t,t_{f}\right]$ is said to have an \textit{admissible set of discontinuities}, if for each $q\in Q$, the discontinuities of the continuous valued feedback control $u_q\left(t,x\right)$ and the discrete valued feedback input $\sigma\left(t,q,x\right)$ are located on lower dimensional manifolds in the time and state space $\mathbb{R} \times \mathbb{R}^{n_{q}}$.
\hfill $\square$
\end{definition}

We note that by A0 an autonomous discrete valued control input $\sigma$ necessarily satisfies the lower dimensional manifold switching set condition of Definition \ref{def:InputWithAdmissibleSetOfDisc} where the sets constitute $C^{\infty}$ submanifolds.

\begin{remark}
For classical (i.e. non-hybrid) systems, a more detailed definition of a feedback control law with an admissible set of discontinuities can be found in \cite[pp.~90--97]{FlemingRishel}. 
The necessary conditions for the Lipschitz continuity of the optimal feedback control are discussed in \cite{KolokoltsovLiYang, GalbraithVinter}, and sufficient conditions for continuity with respect to initial conditions are given in \cite{NgYowThowPEC}.
\end{remark}



\section{The Relationship Between the Hybrid Minimum Principle and Hybrid Dynamic Programming}
\label{sec:HMPvsHDP}

\begin{theorem} \textnormal{\textbf{Evolution of the Cost Sensitivity along a General Trajectory:}}
Consider the hybrid system $\mathbb{H}$ together with the assumptions A0-A2 and the hybrid cost to go \eqref{CostToGoDefinition}. Then for a given hybrid feedback control $I_{L-j+1}^{\left[t,t_{f}\right]} \left(t,q,x\right)$ with an admissible set of discontinuities, the sensitivity function $\nabla J \equiv \frac{\partial}{\partial x} J\left(t,t_{f},q,x,L-j+1;I_{L-j+1}\right)$ satisfies:
\vspace{-1mm}
\begin{equation}
\frac{d}{dt}\nabla J=-\left( \left[ \frac{\partial f_{q}\left(x,u\right)}{\partial x} \right]^T \nabla J +\frac{\partial l_{q}\left(x,u\right)}{\partial x}\right),\label{GradJdynamics}
\end{equation}
subject to the terminal conditions:
\begin{equation}
\nabla J\left(t_{f},q_L,x,0; I_0\right)=\nabla g\left(x\right), \label{GradJterminal}
\end{equation}
and the boundary conditions:
\begin{multline}
\hspace{-6pt} \nabla J\left(t_{j}-,q_{j-1},x,L-j+1;I_{L-j+1}\right)  \hfill
\\
\equiv \nabla J\left(t_{j},q_{j-1},x,L-j+1;I_{L-j+1}\right) \hfill
\\
=\left.\!\nabla\xi_{\sigma_{j}}\right|_{x}^{T}\nabla J\left(t_{j}+,q_{j},\xi_{\sigma_{j}}\left(x\right),L-j;I_{L-j}\right)+p\left.\!\nabla m\right|_{x}+\left.\!\nabla c\right|_{x}, \hfill
\label{GradJboundary}
\end{multline}
with $p=0$ when $\left(t_{j},x\right)\in\mathbb{R}\times\mathbb{R}^{n_{q_{j-1}}}$ belong to a controlled switching set, and
\vspace{-1mm}
\begin{equation}
p=\frac{\left[\nabla J\left(t_{j}+,q_{j},\xi_{\sigma_{j}}\left(x\right),L-j;I_{L-j}\right)\right]^{T}f_{q_{j},\xi}^{\xi,q_{j-1}}+l_{q_{j},\xi}^{q_{j-1}}}{\nabla m^{T}f_{q_{j-1}}\left(x,u\left(t_{j}-\right)\right)} ,
\label{Pequation}
\end{equation}
when $\left(t_{j},x\right)\in\mathbb{R}\times\mathbb{R}^{n_{q_{j-1}}}$ belong to an autonomous switching set, and where in the above equation $f_{q_{j},\xi}^{\xi,q_{j-1}}:\!=\! f_{q_{j}}\!\left(\xi_{\sigma_{j}}\!\!\left(x_{q_{j\!-\!1}}\!\!\left(t_{j}\!-\right)\right),u_{q_{j}}\!\left(t_{j}\right)\right)-\nabla\xi f_{q_{j\!-\!1}}\!\left(x_{q_{j\!-\!1}}\!\!\left(t_{j}\!-\right),u_{q_{j\!-\!1}}\!\!\left(t_{j}\!-\right)\right)$ and $l_{q_{j},\xi}^{q_{j\!-\!1}}\!:\!=\! l_{q_{j}}\!\left(\!\xi_{\sigma_{j}}\!\!\left(\! x_{q_{j\!-\!1}}\!\right),u_{q_{j}}\!\!\left(t_{j}\right)\right)-l_{q_{j\!-\!1}}\!\left(\! x_{q_{j\!-\!1}}\!,u_{q_{j\!-\!1}}\!\left(t_{j}\!-\right)\right)$.
\hfill $\square$
\end{theorem}


\begin{proof}
We first prove that \eqref{GradJdynamics} holds for $t\in\left(t_{L},t_{L+1}\right]\equiv\left(t_{L},t_{f}\right]$ with the terminal condition \eqref{GradJterminal}. Then by assuming that  \eqref{GradJdynamics} holds for $t\in\left(t_{j},t_{j+1}\right]$, $j\leq L$ we show that it also holds for $t\in\left(t_{j-1},t_{j}\right]$ with the boundary condition \eqref{GradJboundary}, with $p=0$ when $t_{j}$ indicates the time of a controlled switching, and with $p$ given by \eqref{Pequation} when $t_{j}\in\tau_{L}$ indicates the time of an autonomous switching. Hence, by mathematical induction, the relation is proved for all $t\in\left[t_{0},t_{f}\right]$.

\textbf{(\textit{\textbf{i}}) No Switching Ahead:} First, consider a Lebesgue time $t\in\left[t_{L},t_{L+1}\right]\equiv\left[t_{L},t_{f}\right]$ and the hybrid trajectory passing through $\left(q_{L},x\right)$, and consider the cost to go \eqref{CostToGoDefinition} for $I_{0}$ which is
\begin{equation}
J\left(t,q_{L},x,0;I_{0}\right)=\int_{t}^{t_{f}}l_{q_{L}}\left(x_{s},u_{s}\right)ds+g\left(x_{f}\right) .
\end{equation}

\vspace{-2mm}
Since by Definition \ref{def:InputWithAdmissibleSetOfDisc} the discontinuities in $x$ of $I_{0}^{\left[t,t_{f}\right]} \equiv u^{\left[t,t_{f}\right]}$ lie on lower dimensional sets which are closed in the induced topology of the space, the partial derivative of $J$ with respect to $x$ exists in an open neighbourhood of $\left(t,x\right)$, and is derived as
\begin{multline}
\frac{\partial J\left(t,q_{L},x,0;I_{0}\right)}{\partial x}=\frac{\partial}{\partial x}\int_{t}^{t_{f}}l_{q_{L}}\left(x_{q_{L}}^{\left(s\right)},u_{q_{L}}^{\left(s\right)}\right)ds+\frac{\partial}{\partial x}g\left(x_{q_{L}}^{\left(t_{f}\right)}\right)
\\
=\int_{t}^{t_{f}}\frac{\partial}{\partial x}l_{q_{L}}\left(x_{s},u_{s}\right)ds+\frac{\partial}{\partial x}g\left(x_{f}\right),
\end{multline}
which is equivalent to
\begin{equation}
\frac{\!\partial J\left(t,q_{L},x,0;I_{0}\right)\!}{\partial x}\!=\!\!\int_{t}^{t_{f}}\!\!\left[\!\!\frac{\partial x_{s}}{\partial x}\!\!\right]^{T}\!\!\!\frac{\partial l_{q_{L}}{\left(x_{s},u_{s}\right)}}{\partial x_{s}}ds\!+\!\left[\!\!\frac{\partial x_{f}}{\partial x}\!\!\right]^{T}\!\!\!\frac{\partial g{\left(x_{f}\right)}}{\partial x_{f}} .
\label{gradJ}
\end{equation}

Taking $t=t_{f}$ the terminal condition for $\frac{\partial J}{\partial x}$ is seen to be determined by
\vspace{-1mm}
\begin{equation}
\frac{\partial J\left(t_{f},q_{L},x,0;I_{0}\right)}{\partial x}=\nabla_{x_{f}}g\left(x_{f}\right)\equiv\nabla g\left(x\right),
\end{equation}
because $x_{f}=x$ when $t=t_{f}$. Hence, \eqref{GradJterminal} is proved. With the notation $x_{s}=\phi_{q_{L}}\left(s,t,x\right)$
and with the smoothness provided by the assumptions A0-A2 for the given control input with an admissible set of discontinuities, we have
\vspace{-1mm}
\begin{multline}
\frac{d}{ds}\left(\frac{\partial}{\partial x}x_{s}\right)=\frac{d}{ds}\left(\frac{\partial}{\partial x}\phi_{q_{L}}\left(s,t,x\right)\right)
\\
=\frac{\partial}{\partial x}\left(\frac{d}{ds}\phi_{q_{L}}\left(s,t,x\right)\right)=\frac{\partial}{\partial x}\left(f_{q_{L}}\left(\phi_{q_{L}}\left(s,t,x\right),u\right)\right),
\end{multline}
from which we obtain
\vspace{-1mm}
\begin{equation}
\frac{d}{ds}\left(\frac{\partial x_{s}}{\partial x}\right)=\left[\frac{\partial f_{q_{L}}}{\partial x_{s}}\right]^{T}\frac{\partial\phi_{q_{L}}\left(s,t,x\right)}{\partial x},\label{GRADphiDynamics}
\end{equation}
with $\frac{\partial\phi_{q_{L}} \! \left(t,t,x\right)}{\partial x} \!= \! I_{n_{q_L} \!\! \times n_{q_L}}$, since $\phi_{q_{L}} \!\! \left(t,t,x\right) \!=\! x$. Let $\Phi_{s,t}^{q_{L}} \! \in \! \mathbb{R}^{n_{q_L} \!\! \times n_{q_L}}$ denote the solution of 
\begin{equation}
\dot{\Phi}_{s,t}^{q_{L}}=\nabla_{x_{s}}f_{q_{L}}\left(x_{s},u_{s}\right)^{T}\Phi_{s,t}^{q_{L}}\equiv\left[\frac{\partial f_{q_{L}}\left(x_{s},u_{s}\right)}{\partial x_{s}}\right]^{T}\Phi_{s,t}^{q_{L}}\,,\label{PhiDynamics}
\end{equation}
with $\Phi_{t,t}^{q_{L}}=I_{n_{q_{L}}\times n_{q_{L}}}$. By the uniqueness of the solutions to \eqref{GRADphiDynamics} and \eqref{PhiDynamics}:
\vspace{-1mm}
\begin{equation}
\frac{\partial}{\partial x}\phi_{q_{L}}\left(s,t,x\right)=\Phi_{s,t}^{q_{L}} \,,\label{StateTransitionGradientEquality}
\end{equation}
for all $x\in\mathbb{R}^{n_{q_{L}}}$. Also by the semi-group property:
\begin{equation}
x=\phi_{q_{L}}\left(s,t,x_{t}\right)=\phi_{q_{L}}\left(s,t,\phi_{q_{L}}\left(t,s,x\right)\right) ,
\end{equation}
and hence by taking the derivative with respect to $x$ we have
\begin{equation}
I_{n_{q_{L}}\times n_{q_{L}}}=\left.\frac{\partial\phi_{q_{L}}\left(s,t,z\right)}{\partial z}\right|_{z=\phi\left(t,s,x_{s}\right)}\frac{\partial\phi_{q_{L}}\left(t,s,x\right)}{\partial x} \,,
\end{equation}
which by \eqref{StateTransitionGradientEquality} is equivalent to
\vspace{-1mm}
\begin{equation}
I_{n_{q_{L}}\times n_{q_{L}}}=\Phi_{s,t}^{q_{L}}\Phi_{t,s}^{q_{L}} \, .\label{PhiInvPhi}
\vspace{-1mm}
\end{equation}

For all $r,s,t \in \left(t_L,t_f\right]$ it is the case that
\vspace{-1mm}
\begin{equation}
\frac{d}{ds}\Phi_{s,r}^{q_{L}}=\left[\frac{\partial f_{q_{L}}\left(x_{s},u_{s}\right)}{\partial x_{s}}\right]^{T}\Phi_{s,r}^{q_{L}} \,, \;\;\;\;\;\;\;\;\Phi_{r,r}^{q_{L}}=I_{n_{q_{L}}\times n_{q_{L}}} \,, \label{dotPhi_s,r}
\end{equation}
\vspace{-3mm}
\begin{multline}
\frac{d}{ds}\left(\Phi_{s,t}^{q_{L}}\Phi_{t,r}^{q_{L}}\right)=\left(\left[\frac{\partial f_{q_{L}}\left(x_{s},u_{s}\right)}{\partial x_{s}}\right]^{T}\Phi_{s,t}^{q_{L}}\right)\Phi_{t,r}^{q_{L}}
\\
=\left[\frac{\partial f_{q_{L}}\left(x_{s},u_{s}\right)}{\partial x_{s}}\right]^{T}\left(\Phi_{s,t}^{q_{L}}\Phi_{t,r}^{q_{L}}\right) ,\label{dotPhis,tPhit,r}
\vspace{-1mm}
\end{multline}
where for \eqref{dotPhis,tPhit,r} at $s=r$ the condition $\Phi_{r,t}^{q_{L}}\Phi_{t,r}^{q_{L}}=I_{n_{q_{L}}\times n_{q_{L}}}$ holds. Hence, from the uniqueness of the solution to the ODEs \eqref{dotPhi_s,r}
and \eqref{dotPhis,tPhit,r} we obtain $\Phi_{s,t}^{q_{L}}\Phi_{t,r}^{q_{L}}=\Phi_{s,r}^{q_{L}}$. Furthermore, \eqref{PhiInvPhi} gives
\vspace{-1mm}
\begin{equation}
0=\frac{d\Phi_{s,t}^{q_{L}}}{dt}\Phi_{t,s}^{q_{L}}+\Phi_{s,t}^{q_{L}}\frac{d\Phi_{t,s}^{q_{L}}}{dt} \,,
\vspace{-1mm}
\end{equation}
and hence
\vspace{-1mm}
\begin{multline}
\frac{d\Phi_{s,t}^{q_{L}}}{dt}=-\Phi_{s,t}^{q_{L}}\frac{d\Phi_{t,s}^{q_{L}}}{dt}\left[\Phi_{t,s}^{q_{L}}\right]^{-1}
\\
=-\Phi_{s,t}^{q_{L}}\left[\frac{\partial f_{q_{L}}\left(x_{t},u_{t}\right)}{\partial x_{t}}\right]^{T} \!\!\!\! \Phi_{t,s}^{q_{L}}\left[\Phi_{t,s}^{q_{L}}\right]^{-1}
=-\Phi_{s,t}^{q_{L}}\left[\frac{\partial f_{q_{L}}\left(x_{t},u_{t}\right)}{\partial x_{t}}\right]^{T} \!\!\!\! . \!\!\!\!\!\!
\vspace{-1mm}
\end{multline}

Differentiating \eqref{gradJ} with respect to $t$ along a trajectory
$\left(q_{L},x\right)$ gives
\vspace{-1mm}
\begin{multline}
\frac{d}{dt}\frac{\partial J\left(t,q_{L},x,0;I_{0}\right)}{\partial x}
\\
=\frac{d}{dt}\!\int_{t}^{t_{f}}\!\!\frac{\partial\phi_{q_{L}}^{\left(s,t,x\right)}}{\partial x}^{T}\!\!\!\!\!\left.\frac{\partial l_{q_{L}}^{\left(z,u_{s}\right)}}{\partial z}\right|_{z=\phi_{q_{L}}\!\!\left(s,t,x\right)}\hspace{-5mm}ds\!+\!\frac{d}{dt}\frac{\partial\phi_{q_{L}}^{\left(t_{f},t,x\right)}}{\partial x}^{T}\!\!\!\!\!\left.\frac{\partial g}{\partial z}\right|_{z=\phi_{q_{L}}\!\!\left(t_{f},t,x\right)}
\\
=-\left\{ \left[\frac{\partial\phi_{q_{L}}\left(s,t,x\right)}{\partial x}\right]^{T}\left.\frac{\partial l_{q_{L}}\left(z,u_{s}\right)}{\partial z}\right|_{z=\phi_{q_{L}}\left(s,t,x\right)}\right\}_{s=t}
\\
+\int_{t}^{t_{f}}\frac{d}{dt}\left\{ \left[\frac{\partial\phi_{q_{L}}\left(s,t,x\right)}{\partial x}\right]^{T}\left.\frac{\partial l_{q_{L}}\left(z,u_{s}\right)}{\partial z}\right|_{z=\phi_{q_{L}}\left(s,t,x\right)}\right\} ds
\\
+\frac{d}{dt}\!\left[\frac{\partial\phi_{q_{L}}\left(t_{f},t,x\right)}{\partial x}\right]^{T}\!\!\!\!\left.\frac{\partial g}{\partial z}\right|_{z=\phi_{q_{L}}\left(t_{f},t,x\right)}\hspace{-5mm}=\!-\!\left\{ \! I_{n\times n}.\frac{\partial l_{q_{L}}\left(x_{t},u_{t}\right)}{\partial x_{t}}\!\right\} 
\\
\! + \!\! \int_{t}^{t_{f}} \!\!\! \left\{ \!\! - \!\! \left[\frac{\partial f_{q_{L}}\left(x_{t},u_{t}\right)}{\partial x_{t}}\right]^{\!T} \!\!\! \left[\frac{\partial\phi_{q_{L}}\left(s,t,x\right)}{\partial x}\right]^{\!T} \!\!\!\!\! \left.\frac{\partial l_{q_{L}}\left(z,u_{s}\right)}{\partial z}\right|_{z=\phi_{q_{L}}\left(s,t,x\right)} \!\!\!\!\!\!\!\!\!\!\!\!\!\! + \! 0 \! \right\} \! ds
\\
+\left\{ \! - \! \left[\frac{\partial f_{q_{L}}\left(x_{t},u_{t}\right)}{\partial x_{t}}\right]^{T} \!\! \left[\frac{\partial\phi_{q_{L}}\left(t_{f},t,x\right)}{\partial x}\right]^{T} \!\!\!\! \left.\frac{\partial g\left(z\right)}{\partial z}\right|_{z=\phi_{q_{L}}\left(t_{f},t,x\right)} \!\!\!\!\! +0\right\} , \!\!\!\!\!\!
\vspace{-1mm}
\end{multline}
where the zero terms arise from
\vspace{-1mm}
\begin{align}
\frac{d}{dt}\left.\nabla_{z}l_{q_{L}}\left(z,u_{s}\right)\right|_{z=\phi_{q_{L}}\left(s,t,x\right)} &=\frac{d}{dt}\nabla_{x_{s}}l_{q_{L}}\left(x_{s},u_{s}\right)=0 ,
\\
\frac{d}{dt}\left.\nabla_{z}g\left(z,u_{s}\right)\right|_{z=\phi_{q_{L}}\left(t_{f},t,x\right)} &=\frac{d}{dt}\nabla_{x_{f}}g\left(x_{f}\right)=0.
\vspace{-1mm}
\end{align}

Hence,
\vspace{-1mm}
\begin{multline}
\frac{d}{dt}\frac{\partial J\left(t,q_{L},x,0;I_{0}\right)}{\partial x}=-\frac{\partial l_{q_{L}}\left(x_{t},u_{t}\right)}{\partial x_{t}} \hfill
\\
\hfill -\left[\!\frac{\partial f_{q_{L}}^{\left(x_{t},u_{t}\right)}}{\partial x_{t}}\!\right]^{T}\!\!\!\left\{ \!\!\int_{t}^{t_{f}}\!\left[\!\frac{\partial x_{s}}{\partial x}\!\right]^{T}\!\frac{\partial l_{q_{L}}^{\left(x_{s},u_{s}\right)}}{\partial x_{s}}ds+\left[\!\frac{\partial x_{f}}{\partial x}\!\right]^{T}\!\!\frac{\partial g\left(x_{f}\right)}{\partial x_{f}}\!\right\} , \hspace{-8pt}
\vspace{-1mm}
\end{multline}
which gives
\vspace{-1mm}
\begin{equation}
\frac{d}{dt}\frac{\partial J\left(t,q_{L},x,0;I_{0}\right)}{\partial x}=-\!\left[\frac{\partial f_{q_{L}}^{\left(x,u\right)}}{\partial x}\right]^{T}\!\!\frac{\partial J\left(t,q_{L},x,0;I_{0}\right)}{\partial x}-\frac{\partial l_{q_{L}}^{\left(x,u\right)}}{\partial x}\,,
\vspace{-1mm}
\end{equation}
with 
\vspace{-1mm}
\begin{equation}
\frac{\partial J}{\partial x} \left(t_{f},q_{L},x,0;I_{0}^{\left[t_{f},t_{f}\right]}\right)=\nabla_{x_{f}}g\left(x_{f}\right)\equiv\nabla_{x}g\left(x\right) .
\end{equation}

\textbf{(\textit{\textbf{ii}}) A Controlled Switching Ahead:} 
Now assume that \eqref{GradJdynamics} holds for $\theta\in\left(t_{j},t_{j+1}\right]$, $j\leq L$ when $t_{j}\in\tau_{L}$ indicates a time of a controlled switching. Then for 
$t_{j-1} < t \leq t_j < \theta \leq t_{j+1}$
\vspace{-1mm}
\begin{multline}
J\left(t,q_{j\!-\!1},x,L\!-\!j\!+\!1;I_{L\!-\!j\!+\!1}^{\left[t,t_{f}\right]}\right)
=\int_{t}^{t_{j}}l_{q_{j-1}}^{\left(x_{s},u_{s}\right)}ds
+c_{\sigma_j}^{\left(x\left(t_{j}-\right)\right)} \hfill
\\
\hfill+\int_{t_{j}}^{\theta}l_{q_{j}}^{\left(x_{\omega},u_{\omega}\right)}d\omega+J\left(\theta,q_{j},x_{\theta},L\!-\! j;I_{L-j}^{\left[\theta,t_{f}\right]}\right),
\label{JatSwitching}
\vspace{-1mm}
\end{multline}
where
\vspace{-1mm}
\begin{equation}
x_{\theta} \! = \! \xi\left(x_{t}\!+\!\int_{t}^{t_{j}-} \!\!\! f_{q_{j-1}}\left(x_{s},u_{s}\right)ds\right) + \int_{t_{j}}^{\theta}f_{q_{j}}\left(x_{\omega},u_{\omega}\right)d\omega.
\label{xWithSwitching}
\vspace{-1mm}
\end{equation}

This gives
\vspace{-1mm}
\begin{multline}
\frac{\partial J {\left(t,q_{j-1},x,L-j+1;I_{L-j+1}\right)}}{\partial x}
=\frac{\partial}{\partial x}\int_{t}^{t_{j}}l_{q_{j-1}}^{\left(x_{s},u_{s}\right)}ds
\\
+\frac{\! \partial c_{\sigma_j}^{\left(x\left(t_{j}-\right)\right)} \! }{\partial x} +\frac{\partial}{ \! \partial x \!} \! \int_{t_{j}}^{\theta}l_{q_{j}}^{\left(x_{\omega},u_{\omega}\right)} d\omega
+\! \frac{\partial J\left( \!\! \theta,q_{j},x_{\theta},L-j;I_{L-j}\right)}{\partial x} , \!\!\!\!\!\!\!\!\!\!
\vspace{-3mm}
\end{multline}
which is equivalent to
\vspace{-3mm}
\begin{multline}
\frac{\partial J\left(t,q_{j-1},x,L-j+1;I_{L-j+1}\right)}{\partial x}
=\int_{t}^{t_{j}}\left[\frac{\partial x_{s}}{\partial x}\right]^{T}\frac{\partial l_{q_{j-1}}^{\left(x_{s},u_{s}\right)}}{\partial x_{s}}ds
\\
+\left[\frac{\partial x_{t_{j}-}}{\partial x}\right]^{T}\frac{\partial c\left(x_{t_{j}-}\right)}{\partial x_{t_{j}-}}
+\int_{t_{j}}^{\theta}\left[\frac{\partial x_{\omega}}{\partial x}\right]^{T}\frac{\partial l_{q_{j}}\left(x_{\omega},u_{\omega}\right)}{\partial x_{\omega}}d\omega
\\
\hfill +\left[\frac{\partial x_{\theta}}{\partial x}\right]^{T}\frac{\partial J\left(\theta,q_{j},x_{\theta},L-j;I_{L-j}\right)}{\partial x_{\theta}}\,,
\label{GradJControlled}
\vspace{-1mm}
\end{multline}
with
\vspace{-3mm}
\begin{multline}
\hspace{-8pt}\frac{\partial x_{\theta}}{\partial x}=\int_{t_{j}}^{\theta}\frac{\partial f_{q_{j}}\left(x_{\omega},u_{\omega}\right)}{\partial x}d\omega+\frac{\partial\xi\left(x_{t}+\int_{t}^{t_{j}}f_{q_{j-1}}\left(x_{s},u_{s}\right)ds\right)}{\partial x}
\\
=\int_{t_{j}}^{\theta}\frac{\partial f_{q_{j}}\left(x_{\omega},u_{\omega}\right)}{\partial x}d\omega+\frac{\partial\xi\left(x_{t_{j}-}\right)}{\partial x},
\vspace{-1mm}
\end{multline}
from which we obtain
\vspace{-1mm}
\begin{equation}
\! \!\!\frac{\partial x_{\theta}}{\partial x} \! = \! \!  \int_{t_{j}}^{\theta}\!\left[\!\frac{\partial x_{\omega}}{\partial x}\!\right]^{\!T} \! \frac{\partial f_{q_{j}}\left(x_{\omega},u_{\omega}\right)}{\partial x_{\omega}}d\omega\!+\!\left[\!\frac{\partial x_{t_{j}-}}{\partial x}\!\right]^{\!T}\!\! \frac{\partial\xi\left(x_{t_{j}-}\right)}{\partial x_{t_{j}-}}\,.
\label{xBeforeAfterControlledSwitching}
\vspace{-1mm}
\end{equation}

In particular, for $x = x_t$ as $t\uparrow t_{j}$ and for $x_{\theta}$ as $\theta\downarrow t_{j}$ equation \eqref{GradJControlled} becomes
\vspace{-1mm}
\begin{multline}
\!\!\!\!\frac{\!\partial J\left(\! t_{j}\!-,q_{j\!-\!1},x_{t_{j}\!-},L\!-\! j\!+\!1;I_{L\!-\! j\!+\!1}\!\right)\!}{\partial x_{t_{j}-}}\!=\!\int_{t_{j}\!-}^{t_{j}}\!\left[\!\!\frac{\partial x_{s}}{\partial x}\!\!\right]^{T}\!\!\frac{\partial l_{q_{j\!-\!1}}{\left(x_{s},u_{s}\right)}\!\!}{\partial x_{s}}ds
\vspace{-1mm}
\\
+\left[\frac{\partial x_{t_{j}-}}{\partial x_{t_{j}-}}\right]^{T}\frac{\partial c\left(x_{t_{j}-}\right)}{\partial x_{t_{j}-}}
+\int_{t_{j}}^{t_{j}+}\left[\frac{\partial x_{s}}{\partial x}\right]^{T}\frac{\partial l_{q_{j}}\left(x_{s},u_{s}\right)}{\partial x_{s}}ds
\vspace{-1mm}
\\
+\left[\frac{\partial x_{t_{j}+}}{\partial x_{t_{j}-}}\right]^{T}\frac{\partial J\left(t_{j}+,q_{j},x_{t_{j}+},L-j;I_{L-j}\right)}{\partial x_{t_{j}+}} \,,
\vspace{-2mm}
\end{multline}
which is equivalent to
\vspace{-2mm}
\begin{multline}
\frac{\partial J\left(t_{j}-,q_{j-1},x_{t_{j}-},L-j+1;I_{L-j+1}\right)}{\partial x_{t_{j}-}}=\frac{\partial c\left(x_{t_{j}-}\right)}{\partial x_{t_{j}-}}
\vspace{-2mm}
\\
+\left[\frac{\partial x_{t_{j}+}}{\partial x_{t_{j}-}}\right]^{T}\frac{\partial J\left(t_{j}+,q_{j},x_{t_{j}+},L-j;I_{L-j}\right)}{\partial x_{t_{j}+}} \,,
\vspace{-1mm}
\end{multline}
and also \eqref{xBeforeAfterControlledSwitching} turns into
\vspace{-1mm}
\begin{equation}
\frac{\partial x_{t_{j}+}}{\partial x_{t_{j}-}}= \! \int_{t_{j}}^{t_{j}+}\left[\frac{\partial x_{\omega}}{\partial x}\right]^{T} \!\! \frac{\partial f_{q_{j}}\left(x_{\omega},u_{\omega}\right)}{\partial x_{\omega}}d\omega+\left[\frac{\partial x_{t_{j}-}}{\partial x_{t_{j}-}}\right]^{T} \!\! \frac{\partial\xi\left(x_{t_{j}-}\right)}{\partial x_{t_{j}-}} \,,
\vspace{-1mm}
\end{equation}
which gives
\vspace{-3mm}
\begin{equation}
\frac{\partial x_{t_{j}+}}{\partial x_{t_{j}-}}=\frac{\partial\xi\left(x_{t_{j}-}\right)}{\partial x_{t_{j}-}}=\left.\nabla\xi\right|_{x_{t_{j}-}}.
\vspace{-1mm}
\end{equation}

Hence, 
\vspace{-1mm}
\begin{multline}
\frac{\partial J\left(t_{j}-,q_{j-1},x_{t_{j}-},L-j+1;I_{L-j+1}^{\left[t_{j},t_{f}\right]}\right)}{\partial x_{t_{j}-}}
\\[-3mm]
=\frac{\partial c\left(x_{t_{j}-}\right)}{\partial x_{t_{j}-}}
+\left.\nabla\xi\right|_{x_{t_{j}-}}^{T}\frac{\partial J\left(t_{j}+,q_{j},x_{t_{j}+},L-j;I_{L-j}^{\left[t_{j}+,t_{f}\right]}\right)}{\partial x_{t_{j}+}} .
\vspace{-1mm}
\end{multline}
and therefore, \eqref{GradJboundary} is shown to hold with $p = 0$ for the controlled switching case. Writing 
\begin{multline}
J\left(t,q_{j-1},x,0;I_{L-j+1}\right)=\int_{t}^{t_{j}}l_{q_{j-1}}\left(x_{s},u_{s}\right)ds
\\
+J\left(t,q_{j-1},x\left(t_j-\right),L-j+1;I_{L-j+1}\right) \,,
\label{JBreakingDown}
\end{multline}
and following a similar procedure as in part $\left(i\right)$ of the proof, equation \eqref{GradJdynamics} is derived for $t\in\left(t_{j-1},t_{j}\right]$.

\textbf{(\textit{\textbf{iii}}) An Autonomous Switching Ahead:} Now assume that \eqref{GradJdynamics} holds for all $\theta\in\left(t_{j},t_{j+1}\right]$, $j\leq L$ when $t_{j}\in\tau_{L}$ indicates a time of an autonomous switching. Then taking the derivative of both sides of the equality \eqref{JatSwitching} with respect to $x$ at $t\in\left(t_{j-1},t_{j}\right]$, with $t_{j-1} < t \leq t_j < \theta \leq t_{j+1}$, yields
\begin{multline}
\!\frac{\partial J\!\left(\!t,q_{j-1},x,L\!-\! j\!+\!1;I_{L-j+1}\!\right)\!}{\partial x}\!=\!\frac{\partial}{\partial x}\!\!\int_{t}^{t_{j}}\!\! l_{q_{j-1}}^{\left(x_{s},u_{s}\right)}\! ds\!+\!\frac{\!\partial c\left(x\left(t_{j}\!-\!\right)\right)\!}{\partial x}
\\
+\frac{\partial}{\partial x}\int_{t_{j}}^{\theta}l_{q_{j}}\left(x_{\omega},u_{\omega}\right)d\omega+\frac{\partial J\left(\theta,q_{j},x_{\theta},L-j;I_{L-j}\right)}{\partial x} \,,
\end{multline}
which gives
\vspace{-1mm}
\begin{multline}
\frac{\partial J\left(t,q_{j-1},x,L-j+1;I_{L-j+1}\right)}{\partial x}
=\frac{\partial t_{j}}{\partial x}\left.l_{q_{j-1}}\left(x_{s},u_{s}\right)\right|_{s=t_{j}-}
\\
+\int_{t}^{t_{j}}\left[\frac{\partial x_{s}}{\partial x}\right]^{T}\frac{\partial l_{q_{j-1}}\left(x_{s},u_{s}\right)}{\partial x_{s}}ds
+\left[\frac{\partial x_{t_{j}-}}{\partial x}\right]^{T}\frac{\partial c\left(x_{t_{j}-}\right)}{\partial x_{t_{j}-}}
\\
-\frac{\partial t_{j}}{\partial x}\left.l_{q_{j}}\left(x_{\omega},u_{\omega}\right)\right|_{\omega=t_{j}}
+\int_{t_{j}}^{\theta}\left[\frac{\partial x_{\omega}}{\partial x}\right]^{T}\frac{\partial l_{q_{j}}\left(x_{\omega},u_{\omega}\right)}{\partial x_{\omega}}d\omega
\\
+\left[\frac{\partial x_{\theta}}{\partial x}\right]^{T}\frac{\partial J\left(\theta,q_{j},x_{\theta},L-j;I_{L-j}\right)}{\partial x_{\theta}} \,,
\vspace{-1mm}
\label{SensitivityAutonomous}
\end{multline}
with the derivative of \eqref{xWithSwitching} derived as
\vspace{-1mm}
\begin{multline}
\frac{\partial x_{\theta}}{\partial x}=\frac{\partial}{\partial x}\xi\left(x_{t}+\int_{t}^{t_{j}}f_{q_{j-1}}^{\left(x_{s},u_{s}\right)}ds\right)-\frac{\partial t_{j}}{\partial x_{t_{j}-}}\left.f_{q_{j}}^{\left(x_{\omega},u_{\omega}\right)}\right|_{\omega=t_{j}}
\\
\!+\!\int_{t_{j}}^{\theta}\!\!\frac{\partial f_{q_{j}}^{\left(x_{\omega},u_{\omega}\right)}}{\partial x}d\omega=\left.\frac{\partial\xi\left(z\right)}{\partial z}\right|_{z\!=\! x_{t}\!+\!\int_{t}^{t_{j}}\!\!\! f_{q_{j-1}}^{\left(x_{s},u_{s}\right)}\! ds\!}^{T}\frac{\partial}{\partial x}\!\left(\! x_{t}\!+\!\int_{t}^{t_{j}}\!\!\!\!\! f_{q_{j-1}}^{\left(x_{s},u_{s}\right)}\! ds\!\right)
\\
-\frac{\partial t_{j}}{\partial x_{t_{j}-}}f_{q_{j}}\left(x_{t_{j}},u_{t_{j}}\right)+\int_{t_{j}}^{\theta}\frac{\partial f_{q_{j}}\left(x_{\omega},u_{\omega}\right)}{\partial x}d\omega \,,
\vspace{-1mm}
\end{multline}
which gives
\vspace{-1mm}
\begin{multline}
\frac{\partial x_{\theta}}{\partial x}=-\frac{\partial t_{j}}{\partial x_{t_{j}-}}f_{q_{j}}\left(x_{t_{j}},u_{t_{j}}\right)+\int_{t_{j}}^{\theta}\frac{\partial f_{q_{j}}\left(x_{\omega},u_{\omega}\right)}{\partial x}d\omega
\\
\!+\! \left.\nabla\xi\right|_{x_{t_{j}^{-}}}\!\!\!\left(\! I_{n\times n} \! +\! \frac{\partial t_{j}}{\partial x_{t_{j}-}}f_{q_{j-1}}^{\left(x_{t_{j}-},u_{t_{j}-}\right)} \!\! + \!\! \int_{t}^{t_{j}}\!\!\frac{\partial x_{s}}{\partial x}^{T}\!\! \frac{\partial f_{q_{j-1}}^{\left(x_{s},u_{s}\right)}}{\partial x_{s}}ds\!\right).
\label{xBeforeAfterAutonomous}
\end{multline}

\vspace{-1mm}
Note that in the above equations, the partial derivative $\frac{\partial t_{j}}{\partial x_{t_{j}-}}$ is not necessarily zero because for $\delta x_{t}\in\mathbb{R}^{n}$ the perturbed trajectory $x_{s}+\delta x_{s}$ arrives on the switching manifold $m$ at a different time $t_{j}^{\prime}-=\left(t_{j}+\delta t_{j}\right)-$, $\delta t_{j}\in\mathbb{R}$. Consider a locally modified control $I_{L-j+1}^{\prime}$ of the form
\vspace{-1mm}
\begin{equation}
I_{L-j+1}^{\prime}=\left(\left(t_{j}+\delta t,\sigma_{j}\right),u^{\prime}\right) ,
\vspace{-1mm}
\end{equation}
with
\vspace{-2mm}
\begin{equation}
u_{s}^{\prime}=\begin{cases}
u_{s} & s\in\left[t,t_{j}\right)\\
u\left(t_{j}-\right) & s\in\left[t_{j},t_{j}+\delta t\right)\\
u_{s} & s\in\left[t_{j}+\delta t,t_{j+1}\right)
\end{cases} ,
\vspace{-1mm}
\end{equation}
if $\delta t\geq0$ and
\vspace{-2mm}
\begin{equation}
u_{s}^{\prime}=\begin{cases}
u_{s} & s\in\left[t,t_{j}+\delta t\right)\\
u\left(t_{j}\right)\equiv u\left(t_{j}+\right) & s\in\left[t_{j}+\delta t,t_{j}\right)\\
u_{s} & s\in\left[t_{j},t_{j+1}\right)
\end{cases} ,
\vspace{-1mm}
\end{equation}
if $\delta t<0$. Since $I_{L-j+1}^{\prime} = I_{L-j+1}$ holds everywhere except only on $\left[t_{j},t_{j}+\delta t\right)$ (or $\left[t_{j}+\delta t,t_{j}\right)$ if $\delta t<0$), the measure of the set of modified controls is of the order $\left|\delta t \right|$. Evidently the perturbed trajectory arrives on the switching manifold when
\vspace{-1mm}
\begin{equation}
m\left(x_{t_{j}+\delta t_{j}-}+\delta x_{t_{j}+\delta t_{j}-}\right)=0.
\vspace{-1mm}
\end{equation}

For $\delta t\geq0$ we may write
\vspace{-1mm}
\begin{equation}
m\left(x_{t_{j}-}+\delta x_{t_{j}-}+\int_{t_{j}}^{t_{j}+\delta t} \hspace{-3mm} f_{q_{j-1}}\left(x_{s}+\delta x_{s},u_{t_{j}-}\right)\right)=m\left(x_{t_{j}-}\right)=0,
\vspace{-1mm}
\end{equation}
which results in
\vspace{-1mm}
\begin{equation}
\! \!\!\! \left[ \! \nabla m \left(x_{t_{j}-}\right)\!\right] ^{T} \!\! \left[\delta x_{t_{j}-}+f_{q_{j-1}}\left(x_{t_{j}-},u_{t_{j}-}\right) \delta t+O\left(\delta t^{2}\right)\right]=0,
\vspace{-1mm}
\end{equation}
or
\vspace{-1mm}
\begin{equation}
\delta t=\frac{-\nabla m^{T}\delta x_{t_{j}-}}{\nabla m^{T}f_{q_{j-1}}\left(x_{t_{j}-},u_{t_{j}-}\right)}+O\left(\delta t^{2}\right).
\label{DeltaTvsDeltaX}
\vspace{-1mm}
\end{equation}

Similarly, for $\delta t<0$ the same result is achieved. 
In particular, as $t\uparrow t_{j}$ and $\theta\downarrow t_{j}$
equation \eqref{SensitivityAutonomous} becomes
\vspace{-1mm}
\begin{multline}
\frac{\partial J\left(t_{j}-,q_{j-1},x_{t_{j}-},L-j+1;I_{L-j+1}\right)}{\partial x_{t_{j}-}}
=\frac{\partial t_{j}}{\partial x_{t_{j}-}}l_{q_{j-1}}^{\left(x_{t_{j}-},u_{t_{j}-}\right)}
\\
+\int_{t_{j}-}^{t_{j}}\left[\frac{\partial x_{s}}{\partial x}\right]^{T}\frac{\partial l_{q_{j-1}}^{\left(x_{s},u_{s}\right)}}{\partial x_{s}}ds
+\left[\frac{\partial x_{t_{j}-}}{\partial x_{t_{j}-}}\right]^{T}\frac{\partial c\left(x_{t_{j}-}\right)}{\partial x_{t_{j}-}}
\\
-\frac{\partial t_{j}}{\partial x_{t_{j}-}}l_{q_{j}}\left(x_{t_{j}},u_{t_{j}}\right)
+\int_{t_{j}}^{t_{j}+}\left[\frac{\partial x_{s}}{\partial x}\right]^{T}\frac{\partial l_{q_{j}}\left(x_{s},u_{s}\right)}{\partial x_{s}}ds
\\
+\left[\frac{\partial x_{t_{j}+}}{\partial x_{t_{j}-}}\right]^{T}\frac{\partial J\left(t_{j}+,q_{j},x_{t_{j}+},L-j;I_{L-j}\right)}{\partial x_{t_{j}+}} \,,
\vspace{-3mm}
\end{multline}
or
\vspace{-1mm}
\begin{multline}
 \! \! \! \! \! \! \! \! \! \frac{\!\partial J\!\left(\! t_{j}\!-,q_{j\!-\!1},x_{t_{j}\!-},L\!-\! j\!+\!1;I_{L\!-\! j\!+\!1}\!\right)}{\partial x_{t_{j}-}}\!=\!\frac{\partial t_{j}}{\!\partial x_{t_{j}-}\!}l_{q_{j-1}}^{\left(x_{t_{j}-}\!,u_{t_{j}-}\!\right)\!\!}\!+\!\frac{\partial c\left(x_{t_{j}\!-}\right)}{\partial x_{t_{j}-}}
\\
\!\!\!\!\! -\frac{\partial t_{j}}{\partial x_{t_{j}-}}l_{q_{j}}^{\left(x_{t_{j}},u_{t_{j}}\right)}
+\left[\frac{\partial x_{t_{j}+}}{\partial x_{t_{j}-}}\right]^{T}\frac{\partial J\left(t_{j}+,q_{j},x_{t_{j}+},L-j;I_{L-j}\right)}{\partial x_{t_{j}+}}, \! \! \! \! \! \! \!
\vspace{-1mm}
\end{multline}
and \eqref{xBeforeAfterAutonomous} turns into
\vspace{-1mm}
\begin{equation}
\frac{\partial x_{t_{j}+}}{\partial x_{t_{j}-}}=\left.\nabla\xi\right|_{x_{t_{j}-}}-\frac{\partial t_{j}}{\partial x_{t_{j}-}}\left(f_{q_{j}}^{\left(x_{t_{j}},u_{t_{j}}\right)}-\left.\nabla\xi\right|_{x_{t_{j}-}}f_{q_{j-1}}^{\left(x_{t_{j}-},u_{t_{j}-}\right)}\right).
\vspace{-1mm}
\end{equation}

Therefore, 
\vspace{-1mm}
\begin{multline}
\!\!\!\!\!\!\!\!\!\!
\frac{\partial J\left(t_{j}-,q_{j-1},x_{t_{j}-},L-j+1;I_{L-j+1}\right)}{\partial x_{t_{j}-}} = \hfill
\\
\!\frac{\!\!\!-\partial t_{j}\!}{\!\!\!\partial x_{t_{j}\!-\!}}\!\!\left(\!\!\!\!\left(l_{q_{j}}\!-\! l_{q_{j-1}}\!\right)\!\!+\!\!\left(\! f_{q_{j}}\!-\!\nabla\xi f_{q_{j-1}}\!\!\right)^{\! T}\!\!\frac{\partial J\left(t_{j}\!+,q_{j},x_{t_{j}\!+},L\!-\! j;I_{L\!-\! j}\right)}{\partial x_{t_{j}+}}\!\!\!\right)\!\!\!\!\!\!
\\
\hfill
+\frac{\partial c\left(x_{t_{j}-}\right)}{\partial x_{t_{j}-}}+\nabla\xi^{T}\frac{\partial J\left(t_{j}+,q_{j},x_{t_{j}+},L-j;I_{L-j}\right)}{\partial x_{t_{j}+}} .
\!\!\!\!\!\!\!
\vspace{-3mm}
\end{multline}

But in the limit as $\delta x_{t_{j}-}\in\mathbb{R}^{n}$ becomes sufficiently small, 
\eqref{DeltaTvsDeltaX} gives
\vspace{-1mm}
\begin{equation}
\frac{\partial t_{j}}{\partial x_{t_{j}-}}=\frac{-\nabla m}{\nabla m^{T}f_{q_{j-1}}\left(x_{t_{j}-},u_{t_{j}-}\right)} \,,
\vspace{-3mm}
\end{equation}
and hence,
\vspace{-1mm}
\begin{multline}
\!\!\!\!\!\!\!\! \frac{ \!\! \partial J\!\left(\! t_{j}\!-,q_{j\!-\!1},x_{t_{j}\!-},L\!\!-\!\! j\!+\!1;I_{L\!-\! j\!\!+\!\!1}\!\right) \!\!\! }{\partial x_{t_{j}-}}\!=\!\nabla\xi^{\!T}\!\frac{ \!\! \partial J\!\left(t_{j}\!+,q_{j},x_{t_{j}\!+},L\!\!-\!\! j;I_{L\!-\! j}\!\right) \!\! }{\partial x_{t_{j}+}}
\\
+\frac{\!\left(l_{q_{j}}\!-\! l_{q_{j-1}}\right)\!+\!\left(f_{q_{j}}\!-\!\nabla\xi f_{q_{j-1}}\right){}^{\! T}\frac{\!\!\partial J\left(t_{j}\!+,q_{j},x_{t_{j}\!+},L\!-\! j;I_{L\!-\! j}\!\right)\!\!}{\partial x_{t_{j}+}}}{\nabla m^{T}f_{q_{j-1}}\left(x_{t_{j}-},u_{t_{j}-}\right)}\nabla m\!+\!\nabla c . \!\!\!\!\!\!\!\!\!
\vspace{-4mm}
\end{multline}

This proves \eqref{GradJboundary} with
\vspace{-3mm}
\begin{equation}
p = \frac{\!\left(l_{q_{j}}\!-\! l_{q_{j-1}}\right)\!+\!\left(f_{q_{j}}\!-\!\nabla\xi f_{q_{j-1}}\right)\!^{T}\underset{x_{t_{j}\!+}}{\nabla}\!\!\! J^{\!\!\left(t_{j}\!+,q_{j},x_{t_{j}\!+},L\!-\! j;I_{L\!-\! j}^{\left[t_{j},,t_{f}\right]}\!\right)}\!}{\nabla m^{T}f_{q_{j-1}}\left(x_{t_{j}-},u_{t_{j}-}\right)},
\end{equation}
which is the same equation for 
$p$ as in \eqref{Pequation}. 
Taking account of \eqref{JBreakingDown}, and following a similar procedure as in part $\left(i\right)$ of the proof, equation \eqref{GradJdynamics} is derived for $t\in\left(t_{j-1},t_{j}\right]$, and as shown above, it is subject to the terminal and boundary conditions \eqref{GradJterminal} and \eqref{GradJboundary} respectively. This completes the proof.
\end{proof}

\begin{theorem}
\label{theorem:HMPHDPrelation}
Consider the hybrid system $\mathbb{H}$ together with the assumptions A0-A2 and the HOCP \eqref{HOCP} for the hybrid cost \eqref{Hybrid Cost}. If there exists an optimal control input with admissible set of discontinuities, then along each optimal trajectory, the adjoint process $\lambda$ in the HMP and  the gradient of the value function $\nabla V$ in the corresponding HDP satisfy the same family of differential equations, almost everywhere, i.e.
\vspace{-1mm}
\begin{align}
\frac{d}{dt}\nabla V &=-\frac{\partial}{\partial x}f_{q^{o}}\left(x^{o},u^{o}\right)^T \nabla V-\frac{\partial}{\partial x}l_{q^{o}}\left(x^{o},u^{o}\right),\label{GradVdynamics}
\\
\frac{d}{dt}\lambda^{o} &=-\frac{\partial}{\partial x}f_{q^{o}}\left(x^{o},u^{o}\right)^T \lambda^{o}-\frac{\partial}{\partial x}l_{q^{o}}\left(x^{o},u^{o}\right),\label{AdjointDynamics}
\end{align}
and satisfy the same terminal and boundary conditions, i.e.
\vspace{-1mm}
\begin{multline}
\nabla V\left(t_{f},q^{o},x^o_{q_L}\left(t_{f}\right),0\right)=\nabla g\left(x_{q_L}^{o}\left(t_{f}\right)\right), \hfill \label{GradVterminal}
\end{multline}
\vspace{-7mm}
\begin{multline}
\nabla V\left(t_{j}-,q^o_{j-1},x^o_{q_{j-1}}\left(t_{j}-\right),L-j+1\right)\\
=\left.\nabla\xi\right|_{x^o_{q_{j-1}}\left(t_{j}-\right)}^{T}\nabla V\left(t_{j}+,q^o_{j},x^o_{q_{j}}\left(t_{j}+\right),L-j\right)
\\
+p\left.\nabla m\right|_{x^o_{q_{j-1}}\left(t_{j}-\right)}+\left.\nabla c\right|_{x^o_{q_{j-1}}\left(t_{j}-\right)},
\label{GradVboundary}
\vspace{-1mm}
\end{multline}
for the gradient of the value function, and
\vspace{-1mm}
\begin{align}
& \lambda^{o}\left(t_{f}\right) =\nabla g\left(x_{q_L}^{o}\left(t_{f}\right)\right), \hfill \label{AdjointTerminal}
\\
& \lambda^{o}\!\left(t_{j}\!-\right)\!=\!\left.\nabla\xi\right|_{\! x_{q_{j\!-\!1}}^{o}\!\left(t_{j}\!-\right)}^{T}\!\lambda^{o}\!\left(t_{j}\!+\right)\!+\! p\left.\nabla m\right|_{\! x_{q_{j\!-\!1}}^{o}\!\left(t_{j}\!-\right)}\!+\!\left.\nabla c\right|_{\! x_{q_{j\!-\!1}}^{o}\!\left(t_{j}\!-\right)},\label{AdjointBoundary}
\vspace{-1mm}
\end{align}
for the adjoint process. Hence, the adjoint process and the gradient of the value function are equal almost everywhere, i.e. 
\vspace{-1mm}
\begin{equation}
\lambda^{o}=\nabla_{x}V\label{lambda=grad(V)}
\vspace{-1mm}
\end{equation}
almost everywhere in the Lebesgue sense on $\bigcup\limits_{i=0}^{L} \left[t_i,t_{i+1}\right] \times \mathbb{R}^{n_{q_i}}$.
\vskip -1mm
\hfill $\square$
\end{theorem}


\begin{proof}
Equations \eqref{AdjointDynamics}, \eqref{AdjointTerminal} and \eqref{AdjointBoundary} are direct results of the Hybrid Minimum Principle in Theorem \ref{theorem:HMP}, and equations \eqref{GradVdynamics}, \eqref{GradVterminal} and \eqref{GradVboundary} hold for the optimal feedback control having an admissible set of discontinuities because equations \eqref{GradJdynamics}, \eqref{GradJterminal} and \eqref{GradJboundary} hold for all feedback controls with admissible sets of discontinuities, including $u^o$ corresponding to $x^o$. Hence, from Theorem \ref{theorem:ExistenceUniqueness} and the resulting uniqueness of the solutions of \eqref{GradVdynamics} and \eqref{AdjointDynamics} that are identical almost everywhere on $t\in\left[t_0,t_f\right]$, it is concluded that \eqref{lambda=grad(V)} holds almost everywhere in the Lebesgue sense on $\bigcup_{i=0}^{L} \left[t_i,t_{i+1}\right] \times \mathbb{R}^{n_{q_i}}$.
\end{proof}



\section{Examples}
\label{sec:Examples}

\subsection*{Example 1} 
Consider a hybrid system with the indexed vector fields:
\vspace{-1mm}
\begin{align}
\dot{x} &=f_{1}\left(x,u\right)=x+x\, u,\label{Ex1f_1}
\\
\dot{x} &=f_{2}\left(x,u\right)=-x+x\, u,\label{Ex1f_2}
\vspace{-1mm}
\end{align}
and the hybrid optimal control problem
\vspace{-1mm}
\begin{multline}
J\left(t_{0},t_{f},h_{0},1;I_1\right)
=\int_{t_0}^{t_{s}}\frac{1}{2}u^{2}dt+\frac{1}{1+\left[x\left(t_{s}-\right)\right]^{2}}
\\
+\int_{t_{s}}^{t_{f}}\frac{1}{2}u^{2}dt+\frac{1}{2}\left[x\left(t_{f}\right)\right]^{2},
\label{Ex1TotalCost}
\vspace{-1mm}
\end{multline}
subject to the initial condition $h_{0}=\left(q\left(t_{0}\right),x\left(t_{0}\right)\right)=\left(q_1,x_{0}\right)$ provided at the initial time $t_{0}=0$. At the controlled switching instant $t_s$, the boundary condition for the continuous state is provided by the jump map $x\left(t_{s}\right) =\xi\left(x\left(t_{s}-\right)\right)=-x\left(t_{s}-\right)$.

\subsubsection*{The HMP Formulation and Results}
Writing down the Hybrid Minimum Principle results for the above HOCP, the Hamiltonians are formed as
\begin{align}
H_{q_1} &=\frac{1}{2}u^{2}+\lambda\, x\left(u+1\right),\label{Ex1H1}
\\
H_{q_2} &=\frac{1}{2}u^{2}+\lambda\, x\left(u-1\right),\label{Ex1H2}
\end{align}
from which the minimizing control input for both Hamiltonian functions is determined as
\begin{equation}
u^{o}=-\lambda x \,. \label{Ex1u^o}
\end{equation}

Therefore, the adjoint process dynamics, determined from \eqref{lambda dynamics} and with the replacement of the optimal control input from \eqref{Ex1u^o}, is written as
\begin{align}
\hspace{-2mm} \dot{\lambda}= \frac{-\partial H_{q_1}}{\partial x} = -\lambda\left(u^{o}+1\right)=\lambda\left(\lambda\, x-1\right), &{}  & t \in \left(t_0,t_s\right) , \label{Ex1Lambda1Dynamics}
\\
\hspace{-2mm} \dot{\lambda}=\frac{-\partial H_{q_2}}{\partial x} =-\lambda\left(u^{o}-1\right)=\lambda\left(\lambda\, x+1\right), &{}  & t \in \left(t_s,t_f\right) , \label{Ex1Lambda2Dynamics}
\end{align}
which are subject to the terminal and boundary conditions
\begin{align}
\lambda\left(t_{f}\right) &=\left.\nabla g\right|_{x\left(t_{f}\right)}=x\left(t_{f}\right),\label{Ex1Lambda(t_f)}
\\
\lambda\left(t_{s}-\right)\equiv\lambda\left(t_{s}\right) &=\left.\nabla\xi\right|_{x\left(t_{s}-\right)}\lambda\left(t_{s}+\right)+\left.\nabla c\right|_{x\left(t_{s}-\right)}
\notag\\
&=-\lambda\left(t_{s}+\right)+\frac{-2x\left(t_{s}-\right)}{\left(1+\left[x\left(t_{s}-\right)\right]^{2}\right)^{2}} \;.
\label{Ex1Lambda(t_s)}
\end{align}

The replacement of the optimal control input \eqref{Ex1u^o} in the continuous state dynamics \eqref{StateDynamics} gives
\begin{align}
\dot{x} &=\frac{\partial H_{q_1}}{\partial \lambda} =x\left(1+u^{o}\right)=-x\left(\lambda\, x-1\right), &{}  & t \in \left(t_0,t_s\right) ,\label{Ex1X1Dynamics}
\\
\dot{x} &=\frac{\partial H_{q_2}}{\partial \lambda}=x\left(-1+u^{o}\right)=-x\left(\lambda\, x+1\right), &{}  & t \in \left(t_s,t_f\right) ,\label{Ex1X2Dynamics}
\end{align}
which are subject to the initial and boundary conditions
\begin{align}
x\left(t_0\right) & =x\left(0\right) =x_{0},\label{Ex1X0}
\\
x\left(t_{s}\right) &=\xi\left(x\left(t_{s}-\right)\right)=-x\left(t_{s}-\right) .\label{Ex1Xs}
\end{align}

The Hamiltonian continuity condition \eqref{Hamiltonian jump} states that
\begin{multline}
H_{q_1}\left(t_{s}-\right)=\frac{1}{2}\left[u^{o}\left(t_{s}-\right)\right]^{2}+\lambda\left(t_{s}-\right)x\left(t_{s}-\right)\left[u^{o}\left(t_{s}-\right)+1\right]
\\
=\frac{1}{2}\left[-\lambda\left(t_{s}-\right)x\left(t_{s}-\right)\right]^{2}
+\lambda\left(t_{s}-\right)x\left(t_{s}-\right)\left[-\lambda\left(t_{s}-\right)x\left(t_{s}-\right)+1\right]
\\
= H_{q_2}\left(t_{s}+\right)
=\frac{1}{2}\left[u^{o}\left(t_{s}+\right)\right]^{2}+\lambda\left(t_{s}+\right)x\left(t_{s}+\right)\left[u^{o}\left(t_{s}+\right)-1\right]
\\
=\frac{1}{2}\left[-\lambda\left(t_{s}+\right)x\left(t_{s}+\right)\right]^{2}
+\lambda\left(t_{s}+\right)x\left(t_{s}+\right)\left[-\lambda\left(t_{s}+\right)x\left(t_{s}+\right)-1\right], \hspace{-3mm}
\end{multline}
which can be written, using \eqref{Ex1Xs}, as
\begin{equation}
x\left(t_{s^-}\right)\left[\lambda\left(t_{s^-}\right)-\lambda\left(t_{s^+}\right)\right]
=\frac{1}{2}\left[x\left(t_{s^-}\right)\right]^{2}\left[\left[\lambda\left(t_{s^-}\right)\right]^{2}-\left[\lambda\left(t_{s^+}\right)\right]^{2}\right] .
\label{Ex1Hcontinuity}
\end{equation}

The solution to the set of ODEs \eqref{Ex1Lambda1Dynamics}, \eqref{Ex1Lambda2Dynamics}, \eqref{Ex1X1Dynamics}, \eqref{Ex1X2Dynamics} together with the initial condition \eqref{Ex1X0} expressed at $t_0$, the terminal condition \eqref{Ex1Lambda(t_f)} determined at $t_f$ and the boundary conditions \eqref{Ex1Xs} and \eqref{Ex1Lambda(t_s)} provided at $t_s$ which is not a priori fixed but determined by the Hamiltonian continuity condition \eqref{Ex1Hcontinuity},  determine the optimal control input and its corresponding optimal trajectory that minimize the cost $J\left(t_{0},t_{f},h_{0},1;I_1\right)$ over $\bm{I_{1}}$, the family of hybrid inputs with one switching. Interested readers are referred to \cite{APPECIFAC2014} in which further steps are taken in order to reduce the above boundary value ODE problem into a set of algebraic equations using the special forms of the differential equations under study.

\subsubsection*{The HDP Formulation and Results}

Theorem \ref{theorem:HDP} states that the value function satisfies the HJB equation \eqref{HJB} almost everywhere. In particular,
\begin{multline}
\hspace{-9pt} -\frac{\partial V\left(t,q_{2},x,0\right)}{\partial t} =\inf_{u}H_{q_{2}}\left(x,\frac{\partial V}{\partial x},u\right) 
\\
=\inf_{u}\left\{ l_{q_{2}}\left(x,u\right)+\frac{\partial V}{\partial x}f_{q_{2}}\left(x,u\right)\right\}
=\inf_{u}\left\{ \frac{1}{2}u^{2}+\frac{\partial V}{\partial x}\left[-x+xu\right]\right\} 
\\
\hspace{12pt}=\left\{ \frac{1}{2}u^{2}+\frac{\partial V}{\partial x}\left[-x+xu\right]\right\}_{u=-x\frac{\partial V}{\partial x}}
=\frac{-1}{2}x^{2}\left(\frac{\partial V}{\partial x}\right)^{2}-x\frac{\partial V}{\partial x}, \hspace{-9pt}
\label{Ex1HJBforQ2}
\end{multline}
and similarly,
\begin{equation}
-\frac{\partial V\left(t,q_{1},x,1\right)}{\partial t}=\frac{-1}{2}x^{2}\left(\frac{\partial V}{\partial x}\right)^{2}+x\frac{\partial V}{\partial x} ,
\label{Ex1HJBforQ1}
\end{equation}
 with the boundary conditions
\begin{equation}
V\left(t_{f},q_2,x,0\right)=g\left(x\left(t_{f}\right)\right)=\frac{1}{2}x^{2},\label{Ex1V2Terminal}
\end{equation}
for $V\left(t,q_2,x,0\right)$, together with 
\begin{equation}
V\left(t_{s},q_1,x,1\right)=\min_{\sigma\in\left\{ \sigma_{q_1 q_2}\right\}}\left\{ V\left(t_{s},q_{2},-x,0\right)+\frac{1}{1+x^{2}} \right\}, \label{Ex1HDPboundary}
\end{equation}
and
\begin{equation}
\frac{-1}{2}x^{2}\left(\frac{\partial V_{q_1}}{\partial x}\right)^{2}+x\frac{\partial V_{q_1}}{\partial x} =\frac{-1}{2}\left(-x\right)^{2}\left(\frac{\partial V_{q_2}}{\partial x}\right)^{2}-\left(-x\right)\frac{\partial V_{q_2}}{\partial x} \; ,
\label{Ex1HDPboundary2}
\end{equation}
which determine $V\left(t,q_1,x,1\right)$ and $t_s$.

\subsubsection*{The HMP - HDP Relationship}
In order to illustrate the results of Theorem \ref{theorem:HMPHDPrelation}, 
we first take the partial derivatives of \eqref{Ex1HJBforQ2} with respect to $x$ to write
\begin{equation}
\frac{\partial}{\partial x}\left(\frac{\partial V}{\partial t}-\frac{1}{2}x^{2}\left(\frac{\partial V}{\partial x}\right)^{2}-x\frac{\partial V}{\partial x}\right)=0,
\end{equation}
or
\begin{equation}
\frac{\partial^{2}V}{\partial x\partial t}-x\left(\frac{\partial V}{\partial x}\right)^{2}-x^{2}\frac{\partial V}{\partial x}\frac{\partial^{2}V}{\partial x^{2}}-\frac{\partial V}{\partial x}-x\frac{\partial^{2}V}{\partial x^{2}}=0 .
\end{equation}

It can easily be verified that the set of states with twice differentiability of  $V\left(t,q_2,x,0\right)$ is $M_{\left(2\right)} = \left(t_s,t_f\right) \times \left(\mathbb{R}-\left\{0\right\}\right)$ which is open dense in $\mathbb{R} \times \mathbb{R}$ and therefore, 
\begin{equation}
\frac{\partial^{2}V}{\partial t\partial x}-x^{2}\frac{\partial V}{\partial x}\frac{\partial^{2}V}{\partial x^{2}}-x\frac{\partial^{2}V}{\partial x^{2}}=x\left(\frac{\partial V}{\partial x}\right)^{2}+\frac{\partial V}{\partial x} .
\label{Ex1V2variations}
\end{equation}

But from the definition of the total derivative, we have
\begin{multline}
\frac{d}{dt}\left(\frac{\partial V}{\partial x}\right)=\frac{\partial^{2}V}{\partial t\partial x}+ \frac{\partial^{2}V}{\partial x^{2}}\,f_{q_{2}\left(x,u^o\right)}
\\
=\frac{\partial^{2}V}{\partial t\partial x}+\frac{\partial^{2}V}{\partial x^{2}}\left(-x^{2}\frac{\partial V}{\partial x}-x\right)
=\frac{\partial^{2}V}{\partial t\partial x}-x^{2}\frac{\partial V}{\partial x}\frac{\partial^{2}V}{\partial x^{2}}-x\frac{\partial^{2}V}{\partial x^{2}} .
\label{Ex1V2differential}
\end{multline}

Therefore, from \eqref{Ex1V2variations} and \eqref{Ex1V2differential}, the governing equation for $\nabla V\left(t,q_2,x,0\right)$ is derived as
\begin{equation}
\frac{d}{dt}\left(\frac{\partial V}{\partial x}\right)=x\left(\frac{\partial V}{\partial x}\right)^{2}+\frac{\partial V}{\partial x}=\frac{\partial V}{\partial x}\left(x\frac{\partial V}{\partial x}+1\right) ,
\label{Ex1NablaV2Dynamics}
\end{equation}
which is the same as the dynamics \eqref{Ex1Lambda2Dynamics} for $\lambda\left(t\right)$, $t\in \left(t_s,t_f\right)$.

Similarly, the differentiation of \eqref{Ex1HJBforQ1} results in
\begin{equation}
\frac{d}{dt}\left(\frac{\partial V}{\partial x}\right)=\frac{\partial V}{\partial x}\left(x\frac{\partial V}{\partial x}-1\right) ,
\label{Ex1NablaV1Dynamics}
\end{equation}
which is the same as the dynamics \eqref{Ex1Lambda1Dynamics} for $\lambda\left(t\right)$, $t\in \left(t_0,t_s\right)$.

The equality of the terminal conditions for $\nabla V\left(t_f,q_2,x,0\right)$ and $\lambda\left(t_f\right)$ becomes obvious by taking the gradient of \eqref{Ex1V2Terminal}, i.e. 
\begin{equation}
\frac{\partial V\left(t_{f},q_{2},x,0\right)}{\partial x}=\frac{\partial g\left(x\right)}{\partial x}=x,
\label{Ex1NablaV2Terminal}
\end{equation}
which is equivalent to \eqref{Ex1Lambda(t_f)}.

Moreover, the equality of the boundary conditions for $\nabla V\left(t_f,q_2,x,0\right)$ and $\lambda\left(t_f\right)$ can be illustrated by taking the gradient of \eqref{Ex1HDPboundary} and writing
\begin{equation}
\frac{\partial}{\partial x}V\left(t_s,q_{1},x,1\right)=\frac{\partial}{\partial x}\left(V\left(t_s,q_{2},-x,0\right)+\frac{1}{1+x^{2}}\right),
\end{equation}
that gives
\begin{equation}
\frac{\partial V\left(t_s,q_{1},x,1\right)}{\partial x}=\left.-\frac{\partial V\left(t_s,q_{2},y,0\right)}{\partial y}\right|_{y=-x}+\frac{-2x}{\left(1+x^{2}\right)^{2}},
\label{Ex1NablaVboundary}
\end{equation}
which is the same boundary condition as the boundary condition \eqref{Ex1Lambda(t_s)} for $\lambda$. Therefore, by the uniqueness of the results of the set of differential equations \eqref{Ex1NablaV2Dynamics} and \eqref{Ex1NablaV1Dynamics} for $\nabla V$ (or equivalently \eqref{Ex1Lambda2Dynamics} and \eqref{Ex1Lambda1Dynamics} for $\lambda$) with the terminal and boundary conditions \eqref{Ex1NablaV2Terminal} and \eqref{Ex1NablaVboundary} for $\nabla V$  (or equivalently \eqref{Ex1Lambda(t_f)} and \eqref{Ex1Lambda(t_s)} for $\lambda$), the gradient of the value function evaluated along every optimal trajectory is equal to the adjoint process corresponding to the same trajectory. Interested readers are referred to \cite{APPECIFAC2014} for further discussion on this example.

\hfill $\square$
\vspace{-6mm}

\subsection*{Example 2}

Consider the hybrid system with the indexed vector fields:
\begin{align}
\dot{x} &=\left[\begin{array}{c}
\dot{x}_{1}\\
\dot{x}_{2}
\end{array}\right]=f_{1}\left(x,u\right)=\left[\begin{array}{c}
x_{2}\\
-x_{1}+u
\end{array}\right],
\label{f1dynamics}
\\
\dot{x} &=\left[\begin{array}{c}
\dot{x}_{1}\\
\dot{x}_{2}
\end{array}\right]=f_{2}\left(x,u\right)=\left[\begin{array}{c}
x_{2}\\
u
\end{array}\right],
\label{f2dynamics}
\end{align}
where autonomous switchings occur on the switching manifold described
by
\vspace{-3mm}
\begin{equation}
m\left(x_{1}\left(t_{s}\right),x_{2}\left(t_{s}-\right)\right)\equiv x_{2}\left(t_{s}-\right)=0,
\label{Ex2SwitchingManifold}
\end{equation}
with the continuity of the trajectories at the switching instant. Consider the hybrid optimal control problem defined as the minimization
of the total cost functional
\begin{equation}
J=\int_{t_{0}}^{t_{f}}\frac{1}{2}u^{2}dt+\frac{1}{2}\left(x_{1}\left(t_{s}-\right)\right)^{2}+\frac{1}{2}\left(x_{2}\left(t_{f}\right)-v_{ref}\right)^{2}.
\end{equation}

\subsubsection*{The HMP Formulation and Results}

Employing the HMP, the corresponding Hamiltonians are defined as
\begin{align}
H_{1} &=\lambda_{1}x_{2}+\lambda_{2}\left(-x_{1}+u\right)+\frac{1}{2}u^{2},
\\
H_{2} &=\lambda_{1}x_{2}+\lambda_{2}u+\frac{1}{2}u^{2} .
\end{align}

The Hamiltonian minimization with respect to $u$ (Eq. \eqref{HminWRTu}) gives
\vspace{-3mm}
\begin{equation}
u^{o}=-\lambda_{2} ,
\end{equation}
for both $q=1$ and $q=2$.

Therefore the state dynamics \eqref{StateDynamics} and the adjoint process dynamics \eqref{lambda dynamics} become
\begin{align}
\dot{x}_{1} &=\frac{\partial H_1}{\partial\lambda_{1}}=x_{2},
\label{Ex2q1x1dynamics}
\\
\dot{x}_{2} &=\frac{\partial H_1}{\partial\lambda_{2}}=-x_{1}+u^{o}=-x_{1}-\lambda_{2},
\label{Ex2q1x2dynamics}
\\
\dot{\lambda}_{1} & =\frac{-\partial H_1}{\partial x_{1}} =\lambda_{2}, \label{Ex2q1lambda1dynamics}
\\
\dot{\lambda}_{2} & =\frac{-\partial H_1}{\partial x_{2}} =-\lambda_{1}, \label{Ex2q1lambda2dynamics}
\end{align}
for $q=1$, and
\begin{align}
\dot{x}_{1} &=\frac{\partial H_2}{\partial\lambda_{1}}=x_{2},
\label{Ex2q2x1dynamics}
\\
\dot{x}_{2} &=\frac{\partial H_2}{\partial\lambda_{2}}=u^{o}=-\lambda_{2},
\label{Ex2q2x2dynamics}
\\
\dot{\lambda}_{1} &=\frac{-\partial H_2}{\partial x_{1}}=0, \label{Ex2q2lambda1dynamics}
\\
\dot{\lambda}_{2} &=\frac{-\partial H_2}{\partial x_{2}}=-\lambda_{1}, \label{Ex2q2lambda2dynamics}
\end{align}
 for $q=2$. 
At the initial time $t=t_0$, the continuous valued states are specified by the initial conditions
\begin{align}
x_{1}\left(t_{0}\right) &=x_{10},
\label{Ex2x10}
\\
x_{2}\left(t_{0}\right) &=x_{20}.
\label{Ex2x20}
\end{align}

At the switching instant $t=t_s$, the boundary conditions for the states and  adjoint processes are determined as
\begin{align}
x_{1}\left(t_{s}\right) &=x_{1}\left(t_{s}-\right) \equiv \lim_{t\uparrow t_{s}}x_{1}\left(t\right),
\label{Ex2x1s}
\\
x_{2}\left(t_{s}\right) &=x_{2}\left(t_{s}-\right) = 0,
\label{Ex2x2s}
\\
\lambda_{1}\left(t_{s}\right) &=\lambda_{1}\left(t_{s}+\right)+\frac{\partial c}{\partial x_{1}}+p\frac{\partial m}{\partial x_{1}}= \lambda_{1}\left(t_{s}+\right)+ x_{1}\left(t_{s}\right), \label{Ex2lambda1boundary}
\\
\lambda_{2}\left(t_{s}\right) &=\lambda_{2}\left(t_{s}+\right)+\frac{\partial c}{\partial x_{2}}+p\frac{\partial m}{\partial x_{2}}=\lambda_{2}\left(t_{s}+\right)+p . \label{Ex2lambda2boundary}
\end{align}

And at the terminal time $t=t_f$, the adjoint processes are determined by \eqref{lambda final condition} as
\begin{align}
\lambda_{1}\left(t_{f}\right) &=\frac{\partial g}{\partial x_{1}}=0, \label{Ex2lambda1terminal}
\\
\lambda_{2}\left(t_{f}\right) &=\frac{\partial g}{\partial x_{2}}=x_{2}\left(t_{f}\right)-v_{ref} . \label{Ex2lambda2terminal}
\end{align}

Note that unlike $t_0$ and $t_f$ which are a priori determined, $t_s$ is not fixed and needs to be determined by the Hamiltonian continuity condition \eqref{Hamiltonian jump} as
\begin{multline}
H_{1}\left(t_{s}-\right)=\lambda_{1}\left(t_{s}-\right)x_{2}\left(t_{s}-\right)-\lambda_{2}\left(t_{s}-\right)x_{1}\left(t_{s}-\right)-\frac{1}{2}\lambda_{2}\left(t_{s}-\right)^{2}
\\
=-\lambda_{2}\left(t_{s}\right)x_{1}\left(t_{s}-\right)-\frac{1}{2}\lambda_{2}\left(t_{s}\right)^{2}
= H_{2}\left(t_{s}+\right)
\\
=\lambda_{1}\left(t_{s}+\right)x_{2}\left(t_{s}+\right)-\frac{1}{2}\lambda_{2}\left(t_{s}+\right)^{2}=-\frac{1}{2}\lambda_{2}\left(t_{s}+\right)^{2},
\end{multline}
i.e.
\begin{equation}
\lambda_{2}\left(t_{s}\right)x_{1}\left(t_{s}-\right)+\frac{1}{2}\lambda_{2}\left(t_{s}\right)^{2}=\frac{1}{2}\lambda_{2}\left(t_{s}+\right)^{2},
\end{equation}
that with the insertion of \eqref{Ex2lambda2boundary}, it becomes
\begin{equation}
\left(\lambda_{2}\left(t_{s}+\right)+p\right)x_{1}\left(t_{s}-\right)+\frac{1}{2}\left(\lambda_{2}\left(t_{s}+\right)+p\right)^{2}=\frac{1}{2}\lambda_{2}\left(t_{s}+\right)^{2} .
\label{Ex2Hcontinuity}
\end{equation}

The set of ODEs \eqref{Ex2q1x1dynamics} to \eqref{Ex2q2lambda2dynamics}, together with the initial conditions \eqref{Ex2x10} and \eqref{Ex2x20} expressed at $t_0$, the boundary conditions \eqref{Ex2x1s}, \eqref{Ex2x2s}, \eqref{Ex2lambda1boundary} and \eqref{Ex2lambda2boundary} provided at $t_s$, and the terminal conditions \eqref{Ex2lambda1terminal} and \eqref{Ex2lambda2terminal} determined at $t_f$, with the two unknowns $t_s$ and $p$ determined by the Hamiltonian continuity condition \eqref{Ex2Hcontinuity} and the switching manifold condition \eqref{Ex2SwitchingManifold}, form an ODE boundary value problem whose solution results in the determination of the optimal control input and its corresponding optimal trajectory that minimize the cost $J\left(t_{0},t_{f},h_{0},1;I_1\right)$ over $\bm{I_{1}}$, the family of hybrid inputs with one switching on the switching manifold  \eqref{Ex2SwitchingManifold}. Interested readers are referred to \cite{APPEC1ADHS2015} for further steps taken in order to reduce the above boundary value ODE problem into a set of algebraic equations using the special forms of the differential equations under study.

\subsubsection*{The HDP Formulation and Results}
For the linear differential equations \eqref{f1dynamics} and \eqref{f2dynamics}, the Hamiltonians for the HJB equation are formed as
\begin{align}
H_{1}\left(x,\nabla V,u\right)&=\frac{1}{2}u^{2}+\frac{\partial V}{\partial x_{1}}\cdot x_{2}+\frac{\partial V}{\partial x_{2}} \cdot \left(-x_{1}+u\right),
\\
H_{2}\left(x,\nabla V,u\right)&=\frac{1}{2}u^{2}+\frac{\partial V}{\partial x_{1}}\cdot x_{2}+\frac{\partial V}{\partial x_{2}}\cdot u,
\end{align}
which have a minimizing control input
\begin{equation}
u^{o}=-{\partial V}/{\partial x_{2}},
\end{equation}
and therefore, the HJB equations are expressed as
\begin{align}
-\frac{\partial V\left(t,q_{2},x,0\right)}{\partial t} &=\frac{-1}{2}\left(\frac{\partial V}{\partial x_{2}}\right)^{2}+x_{2}\frac{\partial V}{\partial x_{1}},
\label{Ex2HJBq2}
\\
-\frac{\partial V\left(t,q_{1},x,1\right)}{\partial t} &=\frac{-1}{2}\left(\frac{\partial V}{\partial x_{2}}\right)^{2}+x_{2}\frac{\partial V}{\partial x_{1}}-x_{1}\frac{\partial V}{\partial x_{2}},
\label{Ex2HJBq1}
\end{align}

The terminal condition at $t=t_f$ is specified as
\begin{equation}
V\left(t_f,q_{2},x,0\right) = \frac{1}{2}\left(x_{2}-v_{ref}\right)^{2},
\label{Ex2Vterminal}
\end{equation}
for $V\left(t,q_2,x,0\right)$, and the boundary condition for $V\left(t,q_1,x,1\right)$ and the switching instant $t=t_s$ are determined by
\begin{equation}
V\left(t_s,q_1,x,1\right) = V\left(t_s,q_2,x,0\right)+ \frac{1}{2} x_{1}^{2} ,
\label{Ex2Voptimality}
\end{equation}
and
\begin{equation}
\frac{-1}{2}\left(\frac{\partial V_{q_1}}{\partial x_{2}}\right)^{2}+x_{2}\frac{\partial V_{q_1}}{\partial x_{1}}-x_{1}\frac{\partial V_{q_1}}{\partial x_{2}} = \frac{-1}{2}\left(\frac{\partial V_{q_2}}{\partial x_{2}}\right)^{2}+x_{2}\frac{\partial V_{q_2}}{\partial x_{1}},
\end{equation}
subject to the switching manifold condition \eqref{Ex2SwitchingManifold}. 

\subsubsection*{The HMP - HDP Relationship}
Similar to Example 1, in order to illustrate the result in Theorem \ref{theorem:HMPHDPrelation}, we shall take partial derivatives of \eqref{Ex2HJBq2} and \eqref{Ex2HJBq1} with respect to $x$. We note that by the definition of the total derivative, \vspace{-1mm}
\begin{equation}
\hspace{-1pt} \frac{d}{dt}\left(\frac{\partial V\left(t,q_{i},x,2-i\right)}{\partial x}\right)
=\frac{\partial^{2}V}{\partial x\partial t}+\frac{\partial^{2}V}{\partial x^{2}}f_{q_{i}}\left(x,-\frac{\partial V}{\partial x_{2}}\right), \hspace{-3pt} \vspace{-1mm}
\end{equation}
which is equivalent to \vspace{-2mm}
\begin{multline}
\frac{d}{dt}\left[\!\!\begin{array}{c}
\frac{\partial V\left(t,q_{2},x,0\right)}{\partial x_{1}}\\
\frac{\partial V\left(t,q_{2},x,0\right)}{\partial x_{2}}
\end{array}\!\!\right]\!=\!\left[\!\!\begin{array}{c}
\frac{\partial^{2}V}{\partial x_{1}\partial t}\\
\frac{\partial^{2}V}{\partial x_{2}\partial t}
\end{array}\!\!\right]\!+\!\left[\!\!\begin{array}{cc}
\frac{\partial^{2}V}{\partial x_{1}^{2}} & \frac{\partial^{2}V}{\partial x_{1}\partial x_{2}}\\
\frac{\partial^{2}V}{\partial x_{2}\partial x_{1}} & \frac{\partial^{2}V}{\partial x_{2}^{2}}
\end{array}\!\!\right]\!\left[\!\!\begin{array}{c}
x_{2}\\
-\frac{\partial V}{\partial x_{2}}
\end{array}\!\!\right]
\\
=\left[\begin{array}{c}
\frac{\partial^{2}V}{\partial x_{1}\partial t}+x_{2}\frac{\partial^{2}V}{\partial x_{1}^{2}}-\frac{\partial^{2}V}{\partial x_{1}\partial x_{2}}\frac{\partial V}{\partial x_{2}}\\
\frac{\partial^{2}V}{\partial x_{2}\partial t}+x_{2}\frac{\partial^{2}V}{\partial x_{2}\partial x_{1}}-\frac{\partial^{2}V}{\partial x_{2}^{2}}\frac{\partial V}{\partial x_{2}}
\end{array}\right],
\label{Ex2totalDerivativeVq2} \vspace{-3mm}
\end{multline}
for $\nabla V\left(t,q_{2},x,0\right)$, and \vspace{-2mm}
\begin{multline}
\frac{d}{dt}\!\left[\!\!\!\begin{array}{c}
\frac{\partial V\left(t_{s},q_{1},x,1\right)}{\partial x_{1}}\\
\frac{\partial V\left(t_{s},q_{1},x,1\right)}{\partial x_{2}}
\end{array}\!\!\!\right]\!\!=\!\!\left[\!\!\begin{array}{c}
\frac{\partial^{2}V}{\partial x_{1}\partial t}\\
\frac{\partial^{2}V}{\partial x_{2}\partial t}
\end{array}\!\!\right]\!\!+\!\!\left[\!\!\begin{array}{cc}
\frac{\partial^{2}V}{\partial x_{1}^{2}} & \!\!\!\frac{\partial^{2}V}{\partial x_{1}\partial x_{2}}\\
\frac{\partial^{2}V}{\partial x_{2}\partial x_{1}}\!\!\! & \frac{\partial^{2}V}{\partial x_{2}^{2}}
\end{array}\!\!\right]\!\left[\!\!\!\begin{array}{c}
x_{2}\\
-x_{1}-\frac{\partial V}{\partial x_{2}}
\end{array}\!\!\!\right]
\\
=\left[\begin{array}{c}
\frac{\partial^{2}V}{\partial x_{1}\partial t}+x_{2}\frac{\partial^{2}V}{\partial x_{1}^{2}}-\frac{\partial^{2}V}{\partial x_{1}\partial x_{2}}\frac{\partial V}{\partial x_{2}}-x_{1}\frac{\partial^{2}V}{\partial x_{1}\partial x_{2}}\\
\frac{\partial^{2}V}{\partial x_{2}\partial t}+x_{2}\frac{\partial^{2}V}{\partial x_{2}\partial x_{1}}-\frac{\partial^{2}V}{\partial x_{2}^{2}}\frac{\partial V}{\partial x_{2}}-x_{1}\frac{\partial^{2}V}{\partial x_{2}^{2}}
\end{array}\right],
\label{Ex2totalDerivativeVq1} \vspace{-1mm}
\end{multline}
for $\nabla V\left(t_{s},q_{1},x,1\right)$. Taking the partial derivative of \eqref{Ex2HJBq2} with respect to $x$ and making a substitution using the resulting equation in \eqref{Ex2totalDerivativeVq2} yields \vspace{-1mm}
\begin{align}
\frac{d}{dt}\left(\frac{\partial V\left(t,q_{2},x,0\right)}{\partial x_{1}}\right) &=0,
\label{Ex2q2gradV1}
\\
\frac{d}{dt}\left(\frac{\partial V\left(t,q_{2},x,0\right)}{\partial x_{2}}\right) &=-\frac{\partial V\left(t,q_{2},x,0\right)}{\partial x_{1}},
\label{Ex2q2gradV2} \vspace{-1mm}
\end{align}
which are equivalent to the differential equations \eqref{Ex2q2lambda1dynamics} and \eqref{Ex2q2lambda2dynamics} for $\lambda\left(t\right)$, $t\in\left(t_{s},t_{f}\right]$. Similarly, taking the partial derivative with respect to $x$ of \eqref{Ex2HJBq1} and making a substitution in \eqref{Ex2totalDerivativeVq1} gives \vspace{-2mm}
\begin{align}
\frac{d}{dt}\left(\frac{\partial V\left(t,q_{1},x,1\right)}{\partial x_{1}}\right) &=\frac{\partial V\left(t,q_{1},x,1\right)}{\partial x_{2}},
\label{Ex2q1gradV1}
\\
\frac{d}{dt}\left(\frac{\partial V\left(t,q_{1},x,1\right)}{\partial x_{2}}\right) &=-\frac{\partial V\left(t,q_{1},x,1\right)}{\partial x_{1}}, 
\label{Ex2q1gradV2} \vspace{-1mm}
\end{align}
which are equivalent to the differential equations \eqref{Ex2q1lambda1dynamics} and \eqref{Ex2q1lambda2dynamics} for $\lambda\left(t\right)$, $t\in\left(t_{0},t_{s}\right]$. Moreover, it can easily be verified that the optimal sensitivity process $\nabla V$ satisfies the terminal condition 
\vspace{-3mm}
\begin{equation}
\nabla V\left(t_{f},q_{2},x,0\right)=\left[\begin{array}{c}
\frac{\partial V\left(t_{f},q_{2},x,0\right)}{\partial x_{1}}\\
\frac{\partial V\left(t_{f},q_{2},x,0\right)}{\partial x_{2}}
\end{array}\right]=\left[\begin{array}{c}
0\\
x_{2}\left(t_{f}\right)-v_{ref}
\end{array}\right],
\label{Ex2gradVf} \vspace{-2mm}
\end{equation}
and the boundary condition \vspace{-2mm}
\begin{multline}
\left[\begin{array}{c}
\frac{\partial V\left(t_{s},q_{1},x,1\right)}{\partial x_{1}}\\
\frac{\partial V\left(t_{s},q_{1},x,1\right)}{\partial x_{2}}
\end{array}\right]=\left[\begin{array}{c}
\frac{\partial V\left(t_{s},q_{2},x,0\right)}{\partial x_{1}}\\
\frac{\partial V\left(t_{s},q_{2},x,0\right)}{\partial x_{2}}
\end{array}\right]+\left[\begin{array}{c}
\frac{\partial c\left(x\left(t_{s}-\right)\right)}{\partial x_{1}}\\
\frac{\partial c\left(x\left(t_{s}-\right)\right)}{\partial x_{2}}
\end{array}\right]
\\
+p\left[\begin{array}{c}
\frac{\partial m\left(x\left(t_{s}-\right)\right)}{\partial x_{1}}\\
\frac{\partial m\left(x\left(t_{s}-\right)\right)}{\partial x_{2}}
\end{array}\right]=\left[\begin{array}{c}
\frac{\partial V\left(t_{s},q_{2},x,0\right)}{\partial x_{1}}+x_{1}\left(t_{s}-\right)\\
\frac{\partial V\left(t_{s},q_{2},x,0\right)}{\partial x_{2}}+p
\end{array}\right],
\label{Ex2gradVs} \vspace{-2mm}
\end{multline}
subject to $x_{2}\left(t_{s}-\right)=0$. Therefore, by the uniqueness of the results of the set of governing differential equations for $\nabla V$ and $\lambda$ which are subject to the same terminal and boundary conditions, along any optimal trajectory, the gradient of the value function is equal to the adjoint process corresponding to the same optimal trajectory. Interested readers are referred to \cite{APPEC1ADHS2015} for further discussion on this example.
\hfill $\square$

\section{Riccati Formalism for Linear Quadratic Tracking Problems}
\label{sec:Riccati}
Consider a hybrid system possessing linear vector fields in the form of \vspace{-1mm}
\begin{equation}
\dot{x}=A_{q_{i}}\left(t\right)x+B_{q_{i}}\left(t\right)u+F_{q_{i}}\left(t\right),\;\;\;\;\;\;\;\; t\in\left[t_{i},t_{i+1}\right), \vspace{-1mm}
\end{equation}
together with a given initial condition $\left(q,x\right) \left(t_0\right) = \left(q_0,x_0\right)$ and linear jump maps \vspace{-1mm}
\begin{equation}
x\left(t_{j}\right)=P_{\sigma_{j}}x\left(t_{j}-\right)+J_{\sigma_{j}}, \vspace{-1mm}
\label{JumpMap}
\end{equation}
provided at the switching instances $t_j, 1 \leq j \leq L$ which are not a priori fixed. If  $t_{j}$ corresponds to an autonomous switching from $q_{j-1}$ to $q_{j}$, the switching manifold constraint $m_{q_{j-1}q_{j}}x\left(t_{j}-\right)+n_{q_{j-1}q_{j}}=0$ is satisfied. A controlled switching instant $t_{j}$, in contrast, is a direct consequence of the discrete control input switching command. Consider the HOCP \vspace{-1mm}
\begin{multline}
\hspace{-8pt} J=\sum_{i=0}^{L}\int_{t_{i}}^{t_{i+1}}\frac{1}{2}\left\Vert x_{q_{i}}\left(t\right)-r_{q_{i}}\left(t\right)\right\Vert_{L_{q_{i}}}^{2}+\frac{1}{2}\left\Vert u_{q_{i}}\left(t\right)\right\Vert_{R_{q_{i}}}^{2}dt
\\
+\sum_{j=1}^{L}\frac{1}{2}\left\Vert x_{q_{j-1}}\left(t_{j}-\right)-d_{q_{j-1}}^{t_{j}}\right\Vert_{C_{\sigma_{j}}}^{2}+\frac{1}{2}\left\Vert x_{q_{L}}\left(t_{f}\right)-d_{q_{L}}^{t_{L+1}}\right\Vert_{G_{q_{L}}}^{2} \hspace{-3pt},\hspace{-5pt} \vspace{-1mm}
\end{multline}
where $L_{q_{i}}^{T}=L_{q_{i}}\geq 0$, $R_{q_{i}}^{T}=R_{q_{i}} >0$, $C_{\sigma_{j}}^{T}=C_{\sigma_{j}} \geq 0$, $G_{q_{L}}^{T}=G_{q_{L}} \geq 0$. For the ease of notation and unless otherwise states, the time varying, continuously differentiable matrices $A_{q}\left(t\right)$, $B_{q}\left(t\right)$, $F_{q}\left(t\right)$, $L_{q}\left(t\right)$ ans $R_{q}\left(t\right)$ are simply denoted by $A_{q}$, $B_{q}$, $F_{q}$, $L_{q}$ and $R_{q}$.

Employing the HMP, and making use of the equivalent HMP-HDP relationship as necessary, the Hamiltonians are formed as \vspace{-1mm}
\begin{multline}
H_{i}=\frac{1}{2}\left(x_{q_{i}}^{\left(t\right)}-r_{q_{i}}^{\left(t\right)}\right)^{T}L_{q_{i}}^{\left(t\right)}\left(x_{q_{i}}^{\left(t\right)}-r_{q_{i}}^{\left(t\right)}\right)+\frac{1}{2}u_{q_{i}}^{\left(t\right)T}R_{q_{i}}^{\left(t\right)}u_{q_{i}}^{\left(t\right)}\\+\lambda_{q_{i}}^{\left(t\right)T}\left(A_{q_{i}}^{\left(t\right)}x_{q_{i}}^{\left(t\right)}+B_{q_{i}}^{\left(t\right)}u_{q_{i}}^{\left(t\right)}+F_{q_{i}}^{\left(t\right)}\right). \hspace{-3pt}
\label{HamiltonianLQT} \vspace{-1mm}
\end{multline}

From Theorem \ref{theorem:HMP}, the Hamiltonian minimization gives \vspace{-1mm}
\begin{equation}
\frac{\partial H_{i}}{\partial u_{q_{i}}}=0\;\Rightarrow\; R_{q_{i}}u_{q_{i}}+B_{q_{i}}^{T}\lambda_{q_{i}}=0\;\Rightarrow\; u_{q_{i}}^{o}=-R_{q_{i}}^{-1}B_{q_{i}}^{T}\lambda^o_{q_{i}}, \vspace{-1mm}
\end{equation}
and hence  \vspace{-1mm}
\begin{multline}
\dot{x}_{q_{i}}^{o}=\frac{\partial H_{i}}{\partial\lambda_{q_{i}}}=A_{q_{i}}x_{q_{i}}^{o}+B_{q_{i}}u_{q_{i}}^{o}+F_{q_{i}}
\\[-\belowdisplayskip]
=A_{q_{i}}x_{q_{i}}^{o}-B_{q_{i}}R_{q_{i}}^{-1}B_{q_{i}}^{T}\lambda_{q_{i}}^{o}+F_{q_{i}}, 
\end{multline}
\vspace{-5mm}
\begin{multline}
\dot{\lambda}_{q_{i}}^{o}=-\frac{\partial H_{i}}{\partial x_{q_{i}}}=-L_{q_{i}}\left(x_{q_{i}}^{o}-r_{q_{i}}\right)-A_{q_{i}}^{T}\lambda_{q_{i}}^{o}
\\[-\belowdisplayskip]
=-L_{q_{i}}x_{q_{i}}^{o}-A_{q_{i}}^{T}\lambda_{q_{i}}^{o}+L_{q_{i}}r_{q_{i}}, \vspace{-1mm}
\end{multline}
which have a matrix representation  \vspace{-1mm}
\begin{equation}
\left[\!\begin{array}{c}
\dot{x}_{q_{i}}^{o}\\
\dot{\lambda}_{q_{i}}^{o}
\end{array}\!\right]=\left[\!\begin{array}{cc}
A_{q_{i}}^{\left(t\right)} & \hspace{-4pt}-B_{q_{i}}^{\left(t\right)}R_{q_{i}}^{\left(t\right)-1}B_{q_{i}}^{\left(t\right)T}\\
-L_{q_{i}}^{\left(t\right)} & -A_{q_{i}}^{\left(t\right)T}
\end{array}\!\right]\left[\!\begin{array}{c}
x_{q_{i}}^{o}\\
\lambda_{q_{i}}^{o}
\end{array}\!\right]+\left[\!\begin{array}{c}
F_{q_{i}}^{\left(t\right)}\\
L_{q_{i}}^{\left(t\right)}r_{q_{i}}^{\left(t\right)}
\end{array}\!\right],\label{TotalAffineProcess} \vspace{-1mm}
\end{equation}
for $t\in\left[t_{i},t_{i+1}\right)$, subject to the boundary conditions \vspace{-1mm} \vspace{-1mm}
\begin{align}
x_{q_{0}}^{o}\left(t_{0}\right)&=x_{0},
\\
x_{q_{j}}^{o}\left(t_{j}\right)&=P_{\sigma_{j}}x_{q_{j-1}}^{o}\left(t_{j}-\right)+J_{\sigma_{j}},
\\
\lambda_{q_{L}}^{o}\left(t_{f}\right)&=\nabla g=G_{q_{L}}\left(x_{q_{L}}^{o}\left(t_{f}\right)-d_{q_{L}}^{t_{L+1}}\right),
\\
\lambda_{\! q_{j\!-\!1}}^{o}\!\left(t_{j}\right)&=P_{\sigma_{j}}^{T}\lambda_{\! q_{j\!-\!1}}^{o}\left(t_{j}\!+\right)+p\, m_{\! q_{j\!-\!1}q_{j}}\!+\! C_{\sigma_{j}}\!\left(x_{\! q_{j\!-\!1}}^{o}\left(t_{j}-\right)\!-d_{\! q_{j\!-\!1}}^{t_{j}}\right). \vspace{-1mm}
\end{align}

Denoting the state transition matrix for the system in \eqref{TotalAffineProcess} by $\varphi$, the solution of \eqref{TotalAffineProcess} can be written as \vspace{-1mm}
\begin{equation}
\left[\!\!\begin{array}{c}
x_{q_{i}}^{o}\!\left(t\right)\\
\lambda_{q_{i}}^{o}\!\left(t\right)
\end{array}\!\!\right]=\varphi_{q_{i}}\!\left(t,t_{i}\right)\left[\!\!\!\begin{array}{c}
x_{q_{i}}^{o}\!\left(t_{i}\right)\\
\lambda_{q_{i}}^{o}\!\left(t_{i}\!+\right)
\end{array}\!\!\!\right]+\int_{t_{i}}^{t}\varphi_{q_{i}}\!\left(t,\tau\right)\left[\!\!\begin{array}{c}
F_{q_{i}}^{\left(\tau\right)}\\
L_{q_{i}}^{\left(\tau\right)}r_{q_{i}}^{\left(\tau\right)}
\end{array}\!\!\right]d\tau, \vspace{-1mm}
\end{equation}
and also as \vspace{-1mm}
\begin{equation}
\!\!\left[\!\!\!\!\begin{array}{c}
x_{q_{i}}^{o}\!\left(t_{i\!+\!1}\!-\right)\\
\lambda_{q_{i}}^{o}\!\left(t_{i\!+\!1}\!-\right)
\end{array}\!\!\!\!\right]\!\!=\!\varphi_{q_{i}}\!\left(t_{i\!+\!1},t\right)\!\left[\!\!\!\!\begin{array}{c}
x_{q_{i}}^{o}\!\left(t\right)\\
\lambda_{q_{i}}^{o}\!\left(t\right)
\end{array}\!\!\!\!\right]\!+\!\int_{t}^{t_{i\!+\!1}}\!\!\!\!\!\varphi_{q_{i}}\!\left(t_{i\!+\!1},\tau\right)\!\left[\!\!\!\!\begin{array}{c}
F_{q_{i}}^{\left(\tau\right)}\\
L_{q_{i}}^{\left(\tau\right)}r_{q_{i}}^{\left(\tau\right)}
\end{array}\!\!\!\!\right]d\tau,\label{AffineSolutionFirstForm} \vspace{-1mm}
\end{equation}
Partitioning $\varphi$ in the form of  \vspace{-1mm}
\begin{equation}
\varphi_{q_{i}}\left(t_{i+1},t\right)=\left[\begin{array}{cc}
\varphi_{q_{i}}^{11}\left(t_{i+1},t\right) & \varphi_{q_{i}}^{12}\left(t_{i+1},t\right)\\
\varphi_{q_{i}}^{21}\left(t_{i+1},t\right) & \varphi_{q_{i}}^{22}\left(t_{i+1},t\right)
\end{array}\right], \vspace{-1mm}
\end{equation}
and denoting \vspace{-1mm}
\begin{equation}
\left[\!\!\!\begin{array}{c}
h_{q_{i}}^{1}\!\left(t\right)\\
h_{q_{i}}^{2}\!\left(t\right)
\end{array}\!\!\!\right]\!\!:=\!\!\int_{t_{i}}^{t}\!\left[\!\!\begin{array}{cc}
\varphi_{q_{i}}^{11}\left(t_{i+1},t\right)\! & \!\varphi_{q_{i}}^{12}\left(t_{i+1},t\right)\\
\varphi_{q_{i}}^{21}\left(t_{i+1},t\right)\! & \!\varphi_{q_{i}}^{22}\left(t_{i+1},t\right)
\end{array}\!\!\right]\!\left[\!\!\begin{array}{c}
F_{q_{i}}^{\left(\tau\right)}\\
L_{q_{i}}^{\left(\tau\right)}r_{q_{i}}^{\left(\tau\right)}
\end{array}\!\!\right]d\tau, \vspace{-1mm}
\end{equation}
turns \eqref{AffineSolutionFirstForm} into \vspace{-1mm}
\begin{align}
\hspace{-8pt}x_{q_{i}}^{o}\!\left(t_{i\!+\!1}-\right)&\!=\!\varphi_{q_{i}}^{11}\!\left(t_{i\!+\!1},t\right)x_{q_{i}}^{o}\!\left(t\right)+\varphi_{q_{i}}^{12}\!\left(t_{i\!+\!1},t\right)\lambda_{q_{i}}^{o}\!\left(t\right)+h_{q_{i}}^{1}\!\left(t\right),\hspace{-4pt}\label{OptimalTrajectoryAdjoint1}
\\
\hspace{-8pt}\lambda_{q_{i}}^{o}\!\left(t_{i\!+\!1}-\right)&\!=\!\varphi_{q_{i}}^{21}\!\left(t_{i\!+\!1},t\right)x_{q_{i}}^{o}\!\left(t\right)+\varphi_{q_{i}}^{22}\!\left(t_{i\!+\!1},t\right)\lambda_{q_{i}}^{o}\!\left(t\right)+h_{q_{i}}^{2}\!\left(t\right).\hspace{-4pt}
\label{OptimalTrajectoryAdjoint2} \vspace{-1mm}
\end{align}

In the location $q_{L}$ with $t\in\left[t_{L},t_{L+1}\right]=:\left[t_{L},t_{f}\right]$
the terminal condition for $\lambda^{o}$ is provided as 
\begin{equation}
\lambda_{q_{L}}^{o}\left(t_{f}\right)=G_{q_{L}}\left(x_{q_{L}}^{o}\left(t_{f}\right)-d_{q_{L}}^{t_{L+1}}\right).\label{LambdaTerminalAffine}
\end{equation}

Replacing $\lambda_{q_{L}}^{o}\left(t_{f}\right)$ from the above equation
in \eqref{OptimalTrajectoryAdjoint2} and substituting
$x_{q_{L}}^{o}\left(t_{f}\right)$ from \eqref{OptimalTrajectoryAdjoint1}
result in \vspace{-1mm}
\begin{multline}
G_{q_{L}}\left(\varphi_{q_{L}}^{11}\!\left(t_{f},t\right)x_{q_{L}}^{o}\left(t\right)+\varphi_{q_{L}}^{12}\!\left(t_{f},t\right)\lambda_{q_{L}}^{o}\left(t\right)+h_{q_{L}}^{1}\!\left(t\right)-d_{q_{L}}^{t_{L+1}}\right)
\\
=\varphi_{q_{L}}^{21}\!\left(t_{f},t\right)x_{q_{L}}^{o}\left(t\right)+\varphi_{q_{L}}^{22}\!\left(t_{f},t\right)\lambda_{q_{L}}^{o}\left(t\right)+h_{q_{L}}^{2}\!\left(t\right), \hspace{-6pt} \vspace{-1mm}
\end{multline}
which is equivalent to \vspace{-1mm}
\begin{multline}
\hspace{-10pt}\left[\! G_{q_{L}}\!\varphi_{q_{L}}^{11}\!\left(t_{f},t\right)\!-\!\varphi_{q_{L}}^{21}\!\left(t_{f},t\right)\!\right]x_{q_{L}}^{o}\!\left(t\right)\!+\! G_{q_{L}}\! h_{q_{L}}^{1}\!\left(t\right)\!-\! G_{q_{L}}\! d_{q_{L}}^{t_{L\!+\!1}}\!\!-\! h_{q_{L}}^{2}\!\left(t\right)
\\
=\left[\varphi_{q_{L}}^{22}\!\left(t_{f},t\right)-G_{q_{L}}\varphi_{q_{L}}^{12}\!\left(t_{f},t\right)\right]\lambda_{q_{L}}^{o}\left(t\right).\vspace{-1mm}
\end{multline}

From the nonsingularity of the coefficients (see e.g. \cite{Kalman})
we may write \vspace{-1mm}
\begin{multline}
\hspace{-9pt}\lambda_{q_{L}}^{o}\!\!\!\left(t\right)\!=\!\left[\!\varphi_{q_{i}}^{22}\!\!\!\left(t_{\! f},t\right)\!-\! G_{q_{L}}\!\varphi_{q_{i}}^{12}\!\!\!\left(t_{\! f},t\right)\!\right]^{\!\!-\!1}\!\!\!\left[\! G_{q_{L}}\!\varphi_{q_{i}}^{11}\!\!\!\left(t_{\! f},t\right)\!-\!\varphi_{q_{i}}^{21}\!\!\!\left(t_{\! f},t\right)\!\right]\! x_{q_{L}}^{o}\!\!\!\left(t\right)
\\
+\!\left[\!\varphi_{q_{i}}^{22}\!\left(t_{f},t\right)\!-\! G_{q_{L}}\!\varphi_{q_{i}}^{12}\!\left(t_{f},t\right)\!\right]^{-1}\!\left[\! G_{q_{L}}\! h_{q_{i}}^{1}\!\left(t\right)\!-\! G_{q_{L}}\! d_{q_{L}}^{t_{L+1}}\!-\! h_{q_{i}}^{2}\!\left(t\right)\!\right].
\label{OriginalRiccatiBasis} \vspace{-1mm}
\end{multline}

With the definition of $K_{q_{L}}\left(t\right)$ and $s_{q_{L}}\left(t\right)$
such that
\begin{equation}
\lambda_{q_{L}}^{o}\left(t\right)=K_{q_{L}}\left(t\right)x_{q_{L}}^{o}\left(t\right)+s_{q_{L}}\left(t\right),\label{LinearAdjointLaw}
\end{equation}
the optimal control law is therefore given by
\begin{equation}
u_{q_{L}}^{o}=-R_{q_{L}}^{-1}B_{q_{L}}^{T}K_{q_{L}}\left(t\right)x_{q_{L}}^{o}\left(t\right)-R_{q_{L}}^{-1}B_{q_{L}}^{T}s_{q_{L}}\left(t\right).
\end{equation}

Differentiation of \eqref{LinearAdjointLaw} gives
\begin{equation}
\dot{\lambda}_{q_{L}}^{o}=\dot{K}_{q_{L}}x_{q_{L}}^{o}+K_{q_{L}}\dot{x}_{q_{L}}^{o}+\dot{s}_{q_{L}}.
\end{equation}

Replacing $\dot{\lambda}_{q_{L}}^{o}$ and $\dot{x}_{q_{L}}^{o}$ from \eqref{TotalAffineProcess} and using \eqref{LinearAdjointLaw} we get
\begin{multline}
\hspace{-7pt} \left[\dot{K}_{q_{L}}+L_{q_{L}}+K_{q_{L}}A_{q_{L}}+A_{q_{L}}^{T}K_{q_{L}}-K_{q_{L}}B_{q_{L}}R_{q_{L}}^{-1}B_{q_{L}}^{T}K_{q_{L}}\right]x_{q_{L}}^{o}
\\
+\left[\dot{s}_{q_{L}}+\left(A_{q_{L}}^{T}-K_{q_{L}}B_{q_{L}}R_{q_{L}}^{-1}B_{q_{L}}^{T}\right)s_{q_{L}}+K_{q_{L}}F_{q_{L}}-L_{q_{L}}r_{q_{L}}\right] =0 . \hspace{-3mm} \label{RiccatiTypeEquation}
\end{multline}
Since the equation \eqref{RiccatiTypeEquation} holds for all $x_{q_{L}}^{o}\in \mathbb{R}^{n_{q_{L}}}$ and $r_{q_{L}}\left(t\right) \in \mathbb{R}^{n_{q_{L}}}$, the Riccati equations \vspace{-1mm}
\begin{align}
\dot{K}_{q_{L}} &=-L_{q_{L}}-K_{q_{L}}A_{q_{L}}-A_{q_{L}}^{T}K_{q_{L}}+K_{q_{L}}B_{q_{L}}R_{q_{L}}^{-1}B_{q_{L}}^{T}K_{q_{L}},
\\
\dot{s}_{q_{L}} &=-\left(A_{q_{L}}^{T}-K_{q_{L}}B_{q_{L}}R_{q_{L}}^{-1}B_{q_{L}}^{T}\right)s_{q_{L}}-K_{q_{L}}F_{q_{L}}+L_{q_{L}}r_{q_{L}}, \vspace{-1mm}
\end{align}
must hold. The terminal conditions can be determined by the evaluation of \eqref{LinearAdjointLaw} at $t_{f}$ and the use of \eqref{LambdaTerminalAffine} \vspace{-1mm}
to get
\begin{align}
K_{q_{L}}\left(t_{f}\right) &=G_{q_{L}}, \label{Kterminal}
\\
s_{q_{L}}\left(t_{f}\right) &=-G_{q_{L}}d_{q_{L}}^{t_{L+1}} . \label{Sterminal} \vspace{-1mm}
\end{align}

At the switching instant $t_{L}$ the adjoint process boundary condition from the HMP is given as \vspace{-1mm}
\begin{equation}
\lambda_{q_{L-1}}^{o}\!\!\left(t_{L}\right)=P_{\sigma_{L}}^{T}\!\lambda_{q_{L}}^{o}\!\!\left(t_{L}\!+\right)+p\, m_{q_{L\!-\!1}q_{L}}\!+C_{\sigma_{L}}\!\left(x_{q_{L-1}}^{o}\!\!\left(t_{L}\!-\right)-d_{q_{L\!-\!1}}^{t_{L}}\right), 
\end{equation}
\vskip-4mm
\noindent
or \vspace{-4mm}
\begin{multline}
K_{q_{L-1}}\left(t_{L}\right)x_{q_{L-1}}^{o}\left(t_{L}-\right)+s_{q_{L-1}}\left(t_{L}\right)
\\
\hspace{45pt}=P_{\sigma_{L}}^{T}\left(K_{q_{L}}\left(t_{L}\right)x_{q_{L-1}}^{o}\left(t_{L}\right)+s_{q_{L}}\left(t_{L}+\right)\right)+p\, m_{q_{L-1}q_{L}}
\\[-\belowdisplayskip]
+C_{\sigma_{L}}\left(x_{q_{L-1}}^{o}\left(t_{L}-\right)-d_{q_{L-1}}^{t_{L}}\right), 
\end{multline}
\vskip-4mm
\noindent which gives \vspace{-3mm}
\begin{multline}
K_{q_{L-1}}\left(t_{L}\right)x_{q_{L-1}}^{o}\left(t_{L}-\right)+s_{q_{L-1}}\left(t_{L}\right)
\\
\hspace{15pt} =\left[P_{\sigma_{L}}^{T}K_{q_{L}}\left(t_{L}\right)P_{\sigma_{L}}+C_{\sigma_{L}}\right]x_{q_{L-1}}^{o}\left(t_{L}-\right)+P_{\sigma_{L}}^{T}s_{q_{L}}\left(t_{L}-\right)
\\
+p\, m_{q_{L-1}q_{L}}-C_{\sigma_{L}}d_{q_{L-1}}^{t_{L}}+P_{\sigma_{L}}^{T}K_{q_{L}}\left(t_{L}\right)J_{\sigma_{L}}.
\label{RiccatiBC} \vspace{-1mm}
\end{multline}

Since \eqref{RiccatiBC} holds for all points on the switching surface, i.e. all $x\left(t_{L}-\right)\in\left\{x \in \mathbb{R}^{n_{q_{L-1}}} : m_{q_{L-1},q_L} x + n_{q_{L-1},q_L} = 0 \right\}$, the following equalities must hold \vspace{-1mm}
\begin{equation}
K_{q_{L-1}}\left(t_{L}\right)=P_{\sigma_{L}}^{T}K_{q_{L}}\left(t_{L}\right)P_{\sigma_{L}}+C_{\sigma_{L}}, \vspace{-1mm}
\label{kRiccatiBC}
\end{equation}
\begin{equation}
\lambda_{q_{L-1}}^{o}\!\!\left(t_{L}\right)=P_{\sigma_{L}}^{T}\!\lambda_{q_{L}}^{o}\!\!\left(t_{L}\!+\right)+p\, m_{q_{L\!-\!1}q_{L}}\!+C_{\sigma_{L}}\!\left(x_{q_{L-1}}^{o}\!\!\left(t_{L}\!-\right)-d_{q_{L\!-\!1}}^{t_{L}}\right). \vspace{-1mm}
\label{sRiccatiBC}
\end{equation}

For the writing of the Hamiltonian continuity condition \eqref{Hamiltonian jump}, we substitute the minimizing control input $u^{o}=-R_{q_{i}}^{-1}B_{q_{i}}^{T}\lambda$ and the (optimal) adjoint process $\lambda^{o}\left(t\right)=K_{q_{L}}\left(t\right)x^{o}\left(t\right)+s_{q_{L}}\left(t\right)$ (see also \eqref{LinearAdjointLaw}) into the Hamiltonian \eqref{HamiltonianLQT}, i.e. \vspace{-2mm}
\begin{multline}
H^o_{i} \left(t, x_{q_i}\right) =\frac{1}{2}\left(x_{q_{i}}-r_{q_{i}}\right)^{T}L_{q_{i}}\left(x_{q_{i}}-r_{q_{i}}\right)
\\[-\belowdisplayskip] 
+\left(K_{q_{i}}\! x_{q_{i}}\!+s_{q_{i}}\right)^{T}\left(A_{q_{i}}\! x_{q_{i}}\!-\frac{1}{2}B_{q_{i}}\! R_{q_{i}}^{-1}\! B_{q_{i}}^{T}\!\left(K_{q_{i}}\! x_{q_{i}}\!+s_{q_{i}}\right)+F_{q_{i}}\!\right).
\label{HamiltonianLQTminimized}
\end{multline}
\vskip-6mm
Substitution of \eqref{HamiltonianLQTminimized} in \eqref{Hamiltonian jump} results in \vspace{-2mm}
\begin{equation}
\begin{array}{c}
\frac{1}{2}\left(x_{q_{L-1}}^{\left(t_{L}-\right)}-r_{_{q_{L-1}}}^{\left(t_{L}-\right)}\right)^{T}L_{q_{L-1}}\left(x_{q_{L-1}}^{\left(t_{L}-\right)}-r_{_{q_{L-1}}}^{\left(t_{L}-\right)}\right)\hfill\\
+\left(K_{_{q_{L-1}}}^{\left(t_{L}-\right)}x_{q_{L-1}}^{\left(t_{L}-\right)}+s_{_{q_{L-1}}}^{\left(t_{L}-\right)}\right)^{T}\left\{ A_{q_{L-1}}x_{q_{L-1}}^{\left(t_{L}-\right)}+F_{q_{L-1}}\vphantom{f^{f^{f}}}\hfill\right.\\
\left.-\frac{1}{2}B_{q_{L-1}}R_{q_{L-1}}^{-1}B_{q_{L-1}}^{T}\left(K_{_{q_{L-1}}}^{\left(t_{L}-\right)}x_{q_{L-1}}^{\left(t_{L}-\right)}+s_{_{q_{L-1}}}^{\left(t_{L}-\right)}\right)\right\} \\
=\frac{1}{2}\left(x_{q_{L}}^{\left(t_{L}\right)}\!-\! r_{_{q_{L}}}^{\left(t_{L}\right)}\right)^{T}\!\! L_{q_{L}}\!\left(x_{q_{L}}^{\left(t_{L}\right)}\!-\! r_{_{q_{L}}}^{\left(t_{L}\right)}\right)+\left(K_{_{q_{L}}}^{\left(t_{L}\right)}x_{q_{L}}^{\left(t_{L}\right)}\!+\! s_{_{q_{L}}}^{\left(t_{L}\right)}\right)^{T}\hfill\\
\hspace{9mm}\left\{ A_{q_{L}}x_{q_{L}}^{\left(t_{L}\right)}-\frac{1}{2}B_{q_{L}}R_{q_{L}}^{-1}B_{q_{L}}^{T}\left(K_{_{q_{L}}}^{\left(t_{L}\right)}x_{q_{L}}^{\left(t_{L}\right)}+s_{_{q_{L}}}^{\left(t_{L}\right)}\right)+F_{q_{L}}\right\} .
\end{array}\vspace{-1mm}
\end{equation}

Using a backward induction and following a similar approach as above, optimal controls are derived as\vspace{-1mm}
\begin{equation}
u_{q_{i}}^{o}=-R_{q_{i}}^{-1}B_{q_{i}}^{T}K_{q_{i}}\left(t\right)x_{q_{i}}^{o}\left(t\right)-R_{q_{i}}^{-1}B_{q_{i}}^{T}s_{q_{i}}\left(t\right), \vspace{-1mm}
\end{equation}
for $t\in \left[t_i,t_{i+1}\right)$, $0\leq i \leq L$, where ${K}_{q_{i}}$ and ${s}_{q_{i}}$ satisfy the following Riccati equations
\begin{align}
\dot{K}_{q_{i}}&=-L_{q_{i}}-K_{q_{i}}A_{q_{i}}-A_{q_{i}}^{T}K_{q_{i}}+K_{q_{i}}B_{q_{i}}R_{q_{i}}^{-1}B_{q_{i}}^{T}K_{q_{i}},
\\
\dot{s}_{q_{i}}&=-\left(A_{q_{i}}^{T}-K_{q_{i}}B_{q_{i}}R_{q_{i}}^{-1}B_{q_{i}}^{T}\right)s_{q_{i}}-K_{q_{i}}F_{q_{i}}+L_{q_{i}}r_{q_{i}},\vspace{-1mm}
\end{align}
$t\in \left[t_i,t_{i+1}\right)$, $0\leq i \leq L$, and are subject to the terminal conditions \eqref{Kterminal} and \eqref{Sterminal}, and the boundary conditions \vspace{-1mm}
\begin{align}
K_{q_{j-1}}\left(t_{j}\right)&=P_{\sigma_{j}}^{T}K_{q_{j}}\left(t_{j}\right)P_{\sigma_{j}}+C_{\sigma_{j}},
\\[-1mm]
s_{q_{j-1}}\left(t_{j}\right)&=P_{\sigma_{j}}^{T}s_{q_{j}}\left(t_{j}+\right)+p\, m_{q_{j-1}q_{j}}-C_{\sigma_{j}}d_{q_{j-1}}^{t_{j}}+P_{\sigma_{j}}^{T}K_{q_{j}}\left(t_{j}\right)J_{\sigma_{j}},
\end{align}
\vskip-3mm 
\noindent
$1 \leq j \leq L$, and where for the minimized Hamiltonian functions \vspace{-6mm}
\begin{multline}
\hspace{-9pt}H^{o}_i\left(t,x_{q_{i}}\right)=\frac{1}{2}\left\Vert x_{q_{i}}-r_{q_{i}}\right\Vert_{L_{q_{i}}}^{2}
\\[-4pt] 
+\!\left(K_{q_{i}}\! x_{q_{i}}\!\!+s_{q_{i}}\!\right)^{T}\!\!\left(\!\! A_{q_{i}}x_{q_{i}}\!\!-\!\frac{1}{2}B_{q_{i}}R_{q_{i}}^{-1}B_{q_{i}}^{T}\!\left(K_{q_{i}}x_{q_{i}}+s_{q_{i}}\!\right)\!+\! F_{q_{i}}\!\!\right)\!,\hspace{-7pt}
\end{multline}
\vskip-3mm 
\noindent satisfy the following Hamiltonian continuity condition at the optimal switching instants \vspace{-1mm}
\begin{equation}
H_{j-1}^{o}\left(t_{j},x_{q_{j-1}}^{\left(t_{j}-\right)}\right)=H_{j}^{o}\left(t_{j},x_{q_{j}}^{\left(t_{j}\right)}\right)\equiv H_{j}^{o}\left(t_{j},P_{\sigma_{j}}x_{q_{j-1}}^{\left(t_{j}-\right)}+J_{\sigma_{j}}\right).
\end{equation} 
\vskip-1mm
Interested readers are referred to \cite{APPECCDC2015} for further discussion and an analytic example on the HMP-HDP relationship.
\hfill $\square$

\vspace{-2mm}

\section{Concluding Remarks}
In this paper it is proved in the context of deterministic hybrid optimal control theory that the adjoint process in the Hybrid Minimum Principle (HMP) and the gradient of the value function in Hybrid Dynamic Programming (HDP) are identical to each other almost everywhere along optimal trajectories; this is because they both satisfy the same Hamiltonian equations with the same boundary conditions. So due to the fact that the same adjoint process - gradient process relationship holds for continuous parameter (i.e. non-hybrid) stochastic optimal control problems via the so-called Feynman-Kac formula (see e.g. \cite{XYZ}), it is natural to expect the adjoint process in the Stochastic HMP \cite{APPECCDC2016} and the gradient of the value function in Stochastic HDP to be identical almost everywhere. Indeed, the formulation of Stochastic Hybrid Dynamic Programming and the investigation of its relationship to the Stochastic Hybrid Minimum Principle is the subject of another study to be presented in a consecutive paper.

\vspace{-2mm}




\bibliographystyle{IEEEtran}
\bibliography{APPEC2016ArXiv}

\appendix

\subsection{Proof of Theorem \ref{theorem:VisLipschitz}}
\label{sec:AppendixVisLipschitz}

For simplicity of notation we use $x, \hat{x}, u$, etc. instead of $x_q, \hat{x}_q, u_q$, etc. whenever the spaces $\mathbb{R}^{n_q}, \mathbb{R}\times \mathbb{R}^{n_q}, U_q$, etc. can easily be distinguished. For a given hybrid control input $I_{L-j+1}=\left(S_{L-j+1},u\right)$ we use $\hat{x}_{\tau} \equiv \hat{x} \left(\tau;t,\hat{x}_t\right)$ to denote the extended continuous valued state as in \eqref{ExtendedState} at the instant $\tau$ passing through $x_t$, where $t \leq \tau \leq t_f$. We also define
\begin{equation}
K_1=\sup\left\{ \left\Vert \hat{f}_{q}\left(\hat{x},u\right)\right\Vert :\left(q,\hat{x},u\right)\in Q\times \hat{B}_{r}\times U\right\}, 
\end{equation}
where $\hat{B}_{r}:=\left\{\hat{x}=\left[z, x^T\right]^T: \left|z\right|^2 + \left\Vert x\right\Vert^2 <r^2\right\}$.
\begin{proof}
First, consider the stage where no remaining switching is available and hence $t\in \left(t_L,t_{L+1}\right) = \left(t_L,t_{f}\right)$. In this case \vspace{-1mm}
\begin{equation}
\hat{x}\left(t_{f};t,\hat{x}_{t}\right)=\hat{x}_{t}+\int_{t}^{t_{f}}\hat{f}_{q_{L}}\left(\hat{x}_{\tau},u_{\tau}\right)d\tau,\vspace{-1mm}
\end{equation}
which gives\vspace{-1mm}
\begin{equation}
\left\Vert \hat{x}\left(t_{f};t,\hat{x}_{t}\right)-\hat{x}_{t}\right\Vert \leq K_{1}\left|t_{f}-t\right|+\int_{t}^{t_{f}}\hat{K}_{f}\left\Vert \hat{x}\left(\tau;t,\hat{x}_{t}\right)-\hat{x}_{t}\right\Vert d\tau,
\end{equation}
where $\hat{K}_{f}$ depends only on $K_f$ and $K_l$ which are defined in assumptions A0 and A2 respectively. By the Gronwall-Bellman inequality this results in\vspace{-1mm}
\begin{multline}
\left\Vert \hat{x}\left(t_{f};t,\hat{x}_{t}\right)-\hat{x}_{t}\right\Vert \leq K_{1}\left|t_{f}-t\right|+\int_{t}^{t_{f}}\hat{K}_{f}K_{1}\left(\tau-t\right)e^{\hat{K}_{f}\left(t_{f}-\tau\right)}d\tau
\\
\leq K_{2}\left|t_{f}-t\right|\leq K_{2}\left|t_{f}-t_{L}\right|,\vspace{-1mm}
\label{NoSwitchBellmanGronwall}
\end{multline}
where $K_{2}=\max\left\{ K_{1},\hat{K}_{f}K_{1}\left(t_{f}-t_{L}\right)e^{\hat{K}_{f}\left(t_{f}-t_{L}\right)}\right\}$.
Hence, by the semi-group properties of ODE solutions and by use of \eqref{NoSwitchBellmanGronwall}, for $s \geq t$ and $\hat{x}_s \in N_{r_{\hat{x}}}\left(\hat{x}_t\right)$ we have
\begin{multline}
\left\Vert \hat{x}\left(t_{f};t,\hat{x}_{t}\right)-\hat{x}\left(t_{f};s,\hat{x}_{s}\right)\right\Vert
 \leq\left\Vert \hat{x}_{t}-\hat{x}_{s}\right\Vert +\left\Vert \hat{x}\left(s;t,\hat{x}_{t}\right)-\hat{x}_{t}\right\Vert 
\\
+\int_{s}^{t_{f}}\hat{K}_{f}\left\Vert \hat{x}\left(\tau;t,\hat{x}_{t}\right)-\hat{x}\left(\tau;s,\hat{x}_{s}\right)\right\Vert d\tau
\\
\leq\left\Vert \hat{x}_{t}-\hat{x}_{s}\right\Vert +K_{2}\left|s-t\right|+\int_{s}^{t_{f}}\hat{K}_{f}\left\Vert \hat{x}\left(\tau;t,\hat{x}_{t}\right)-\hat{x}\left(\tau;s,\hat{x}_{s}\right)\right\Vert d\tau ,\vspace{-1mm}
\end{multline}
and therefore, by the Gronwall inequality we have\vspace{-1mm}
\begin{multline}
\left\Vert \hat{x}\left(t_{f};t,\hat{x}_{t}\right)-\hat{x}\left(t_{f};s,\hat{x}_{s}\right)\right\Vert \leq\left(\left\Vert \hat{x}_{t}-\hat{x}_{s}\right\Vert +K_{2}\left|s-t\right|\right)e^{\hat{K}_{f}\left(t_{f}-s\right)}
\\
\leq\left(\left\Vert \hat{x}_{t}-\hat{x}_{s}\right\Vert +K_{2}\left|s-t\right|\right)e^{\hat{K}_{f}\left(t_{f}-t_{L}\right)}\leq K\left(\left\Vert \hat{x}_{t}-\hat{x}_{s}\right\Vert ^{2}+\left|s-t\right|^{2}\right)^{\frac{1}{2}} \! \!\! , \hspace{-4mm}
\end{multline}
for some $K<\infty$ which depends only on $t_f - t_L$, $K_1$ and $\hat{K}_f$ and not on the control input.

Since $\hat{g}$ is Lipschitz in $\hat{x}$ and $\hat{x} \left(t_f;t,\hat{x}_t\right)$ is Lipschitz in $\left(t,\hat{x}_t\right) \equiv \left(t,\left[z_t, x_t^T\right]^T \right)$, the performance function
\begin{equation}
J\left(t,t_{f},q,x,0;I_{0}\right) = \hat{g} \left(\hat{x} \left(t_f;t,\hat{x}_t\right)\right) \equiv \!\! \int_{t}^{t_{f}} \!\! l_{q}\left(x,u\right)ds+g\left(x_{q_{L}} \! \left(t_{f}\right)\right) ,
\end{equation}
is Lipschitz in $x \in B_r$ (and $\hat{x} \in \hat{B}_r$) uniformly in $t\in \left(t_L,t_f\right)$ with a Lipschitz constant independent of the control. Further, since the infimum of a family of Lipschitz functions with a common Lipschitz constant is also Lipschitz with the same Lipschitz constant, $V\left(t,t_{f},q,x,0\right)$, the value function with no switchs remaining, is Lipschitz in $x \in B_r$ uniformly in $t\in \left(t_L,t_f\right)$.

Now consider $t,s\in \left(t_j,t_{j+1}\right)$ where $t_{j+1}$ indicates a time of an autonomous switching for the trajectory $\hat{x}\left(\tau;t,\hat{x}_{t}\right)$, and consider for definiteness the case where $\hat{x}\left(\tau;s,\hat{x}_{s}\right)$ arrives on the switching manifold described locally by $m\left(x\right) = 0$ at a later time $t_{j+1}+\delta t$ (the case with an earlier arrival time can be handled similarly by considering $\delta t<0$). It directly follows by replacing $\hat{f}_{q_L}$ and $t_{f}$ by $\hat{f}_{q_j}$ and $t_{j+1}-$ in the above arguments, that
\begin{equation}
\left\Vert \hat{x}\left(t_{j+1}-;t,\hat{x}_{t}\right)-\hat{x}\left(t_{j+1}-;s,\hat{x}_{s}\right)\right\Vert \leq K^{\prime}\left(\left\Vert \hat{x}_{t}-\hat{x}_{s}\right\Vert ^{2}+\left|s-t\right|^{2}\right)^{\frac{1}{2}} \!\! .
\end{equation}

Now since
\begin{multline}
\left\Vert \hat{x}\left(t_{j+1}+\delta t-;s,\hat{x}_{s}\right)-\hat{x}\left(t_{j+1}-;s,\hat{x}_{s}\right)\right\Vert
\\
\leq K_2 \left|t_{j+1}+\delta t - t_{j+1}\right| =  K_2 \left|\delta t\right|,
\end{multline}
and
\begin{multline}
\left\Vert \hat{x}\left(t_{j+1}+\delta t-;s,\hat{x}_{s}\right)-\hat{x}\left(t_{j+1}-;t,\hat{x}_{t}\right)\right\Vert^2
\\
\leq \left\Vert \hat{x}\left(t_{j+1}+\delta t-;s,\hat{x}_{s}\right)-\hat{x}\left(t_{j+1}-;s,\hat{x}_{s}\right)\right\Vert^2 
\\
+ \left\Vert \hat{x}\left(t_{j+1}-;t,\hat{x}_{t}\right)-\hat{x}\left(t_{j+1}-;s,\hat{x}_{s}\right)\right\Vert^2 ,
\end{multline}
it is sufficient to show that the upper bound for $\left|\delta t\right|$ is proportional to $\left(\left\Vert \hat{x}_{t}-\hat{x}_{s}\right\Vert ^{2}+\left|s-t\right|^{2}\right)^{\frac{1}{2}}$. This can be shown to hold by considering the fact that
\vspace{-1mm}
\begin{multline}
m\left(x\left(t_{j+1}+\delta t-;s,x_{s}\right)\right)
\\
=m\left(x\left(t_{j+1}-;s,x_{s}\right)+\int_{t_{j}}^{t_{j}+\delta t}f_{q_{j}}\left(x\left(\tau;s,x_{s}\right),u_{t_{j}-}\right)d\tau\right)
\\
\hspace{3pt}=m\left(\! x\left(t_{j\!+\!1}^{\;\;-};t,x_{t}\right)\!+\!\delta x\left(t_{j\!+\!1}^{\;\;-}\right)\!+\!\int_{t_{j}}^{t_{j}\!+\!\delta t}\!\!\! f_{q_{j}}^{\left(x\left(\tau;s,x_{s}\right),u_{t_{j}^{\;-}}\right)}\!\! d\tau\!\right)
\\
=m\left(x\left(t_{j+1}-;t,x_{t}\right)\right)=0 .
\vspace{-1mm}
\end{multline}

For $\left\Vert \delta x\left(t_{j+1}-\right) \right\Vert < \epsilon_{j+1}$ sufficiently small, 
\vspace{-1mm}
\begin{equation}
\nabla m^{T}\left(\delta x_{t_{j+1}-}+\int_{t_{j}}^{t_{j}+\delta t}f_{q_{j}}^{\left(x\left(\tau;s,x_{s}\right),u_{t_{j}-}\right)}d\tau\right)+O\left(\epsilon_{j+1}^{2}\right)=0,
\vspace{-1mm}
\end{equation}
which is equivalent to
\vspace{-1mm}
\begin{equation}
\nabla m^{T}\delta x\left(t_{j+1}-\right)+\int_{t_{j}}^{t_{j}+\delta t}\nabla m^{T}f_{q_{j}}^{\left(x\left(\tau;s,x_{s}\right),u_{t_{j}-}\right)}d\tau
+O\left(\epsilon_{j+1}^{2}\right)=0.
\vspace{-1mm}
\end{equation}

Due to the transversal arrival of the trajectories with respect to the smooth switching manifold, $\left|\nabla m^T f_{q_j}\right|$ is lower bounded by a strictly positive number $k_{m,f}$ (see \eqref{TransversalityOfTrajectoriesToManifolds}) and hence,
\vspace{-1mm}
\begin{multline}
\hspace{-3mm}\left|\nabla m^{T}\delta x\left(t_{j+1}-\right)+O\left(\epsilon_{j+1}^{2}\right)\right|
=\left|\int_{t_{j}}^{t_{j}+\delta t}\nabla m^{T}f_{q_{j}}^{\left(x\left(\tau;s,x_{s}\right),u_{t_{j}-}\right)}d\tau\right|
\\
\geq\int_{t_{j}}^{t_{j}+\delta t}\left|\nabla m^{T}f_{q_{j}}^{\left(x\left(\tau;s,x_{s}\right),u_{t_{j}-}\right)}\right|d\tau
\geq k_{m,f}\left|\delta t\right|,
\vspace{-1mm}
\end{multline}
which gives
\vspace{-1mm}
\begin{multline}
\left|\delta t\right|\leq\frac{1}{k_{m,f}}\left(\left\Vert \nabla m\right\Vert \left\Vert \delta x\left(t_{j+1}-\right)\right\Vert +\left|O\left(\epsilon_{j+1}^{2}\right)\right|\right)
\\
\leq\frac{1}{k_{m,f}}\left\Vert \nabla m\right\Vert \epsilon_{j+1}+\epsilon_{j+1}
\leq\left(\frac{\left\Vert \nabla m\right\Vert}{k_{m,f}} +1\right)\epsilon_{j+1}=K_{j+1}\epsilon_{j+1} .
\vspace{-1mm}
\end{multline}

Hence, for $t\in\left(t_j,t_{j+1}\right)$ and $x_t\in B_r$ there exist a neighbourhood $N_{r_x}\left(x_t\right)$ such that for $s\in\left(t_j,t_{j+1}\right)$ and $x_s \in {\cal N}_{r_x}\left(x_t\right)$ we have $\left\Vert \delta x\left(t_{j+1}-\right) \right\Vert \leq K^{\prime}\left(\left\Vert \hat{x}_{t}-\hat{x}_{s}\right\Vert ^{2}+\left|s-t\right|^{2}\right)^{\frac{1}{2}} < \epsilon_{j+1}$ in order to ensure that $\delta t \leq K_{j+1} \epsilon_{j+1}$ and consequently
\vspace{-1mm}
\begin{equation}
\!\left\Vert \hat{x}\!\left(t_{j+\!1}\!+\!\delta t\!-\!;s,\hat{x}_{s}\!\right)\!-\!\hat{x}\!\left(t_{j+\!1}\!-\!;t,\hat{x}_{t}\!\right)\right\Vert \!\leq\! K\!\left(\!\left\Vert \hat{x}_{t}\!-\!\hat{x}_{s}\right\Vert ^{2}\!\!+\!\left|s\!-\! t\right|\!^{2}\!\right)\!^{\frac{1}{2}},
\vspace{-1mm}
\end{equation}
for $K$ independent of the control. Since $\hat{\xi}$ is smooth and time invariant, it is therefore Lipschitz in $\hat{x}$ uniformly in time. Since $c\left(x\left(t_{J+1}\right)\right)$ is embedded in $\hat{\xi}$ (see \eqref{ExtendedJump2}), at the switching time $t_{j+1}$ we have
\vspace{-1mm}
\begin{equation}
J\!\left(t_{j+1}\!-\!,q_{j},\hat{x},L\!-\! j;I_{L\!-\! j}\!\right)\!=\! J\!\left(t_{j\!+\!1},q_{j\!+\!1},\hat{\xi}\left(\hat{x}\right),L\!-\! j\!-\!1;I_{L\!-\! j\!-\!1}\!\right) , \vspace{-1mm}
\end{equation}
the Lipschitz property for the cost to go function $J\left(t_{j+1}-,q_{j},\hat{x},L-j;I_{L-j}\right)$ follows from the smoothness of $\hat{\xi}$ and the Lipschitz property of $J\left(t,q_{j+1},\hat{x}_t,L-j-1;I_{L-j-1}\right)$. Namely, by backward induction from the Lipschitzness of $J\left(t,q_{L},\hat{x}_t,0;I_{o}\right)$ proved earlier, it is concluded that $J\left(t,q_{L-1},\hat{x}_t,1;I_{1}\right)$ is Lipschitz, from which $J\left(t,q_{L-2},\hat{x}_t,1;I_{2}\right)$ is concluded to be Lipschitz, etc. Since the Lipschitz constant is independent of the control and because the infimum of a family of Lipschitz functions with a common Lipschitz constant is also Lipschitz with the same Lipschitz constant, \eqref{VisLipschitz} holds and hence, the value function is Lipschitz.
\end{proof}

\vspace{-12mm}

\begin{IEEEbiography}[{\includegraphics[width=1in,height=1.25in,clip,keepaspectratio]{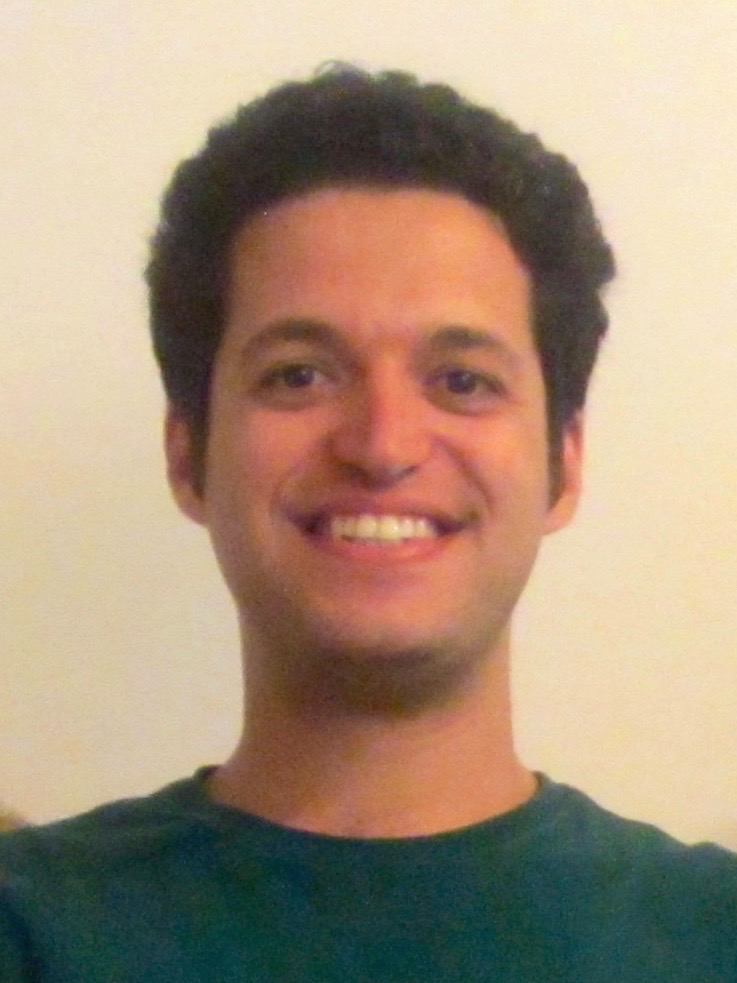}}]%
{Ali Pakniyat}
received the B.Sc. degree in 2008 from the Department of Mechanical Engineering, Shiraz University, Shiraz, Iran, and the M.Sc. degree in Applied Mechanics and Design from the School of Mechanical Engineering, Sharif University of Technology, Tehran, Iran, in 2010. He received the Ph.D. degree in 2016 from the Department of Electrical and Computer Engineering, McGill University, Montreal, Canada. His research interests include \text{deterministic} and stochastic optimal control, \text{nonlinear} and hybrid systems, analytical mechanics and chaos, with applications in automotive industry, sensors and actuators, and robotics.
\end{IEEEbiography}
\vspace{-12mm}
\begin{IEEEbiography}[{\includegraphics[width=1in,height=1.25in,clip,keepaspectratio]{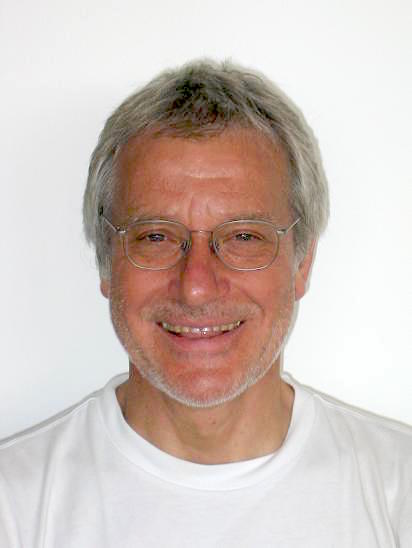}}]%
{Peter E. Caines}
received the BA in mathematics from Oxford University in 1967 and the PhD in systems and control theory in 1970 from Imperial College, University of London, under the supervision of David Q. Mayne, FRS. After periods as a postdoctoral researcher and faculty member at UMIST, Stanford, UC Berkeley, Toronto and Harvard, he joined McGill University, Montreal, in 1980, where he is James McGill Professor and Macdonald Chair in the Department of Electrical and Computer Engineering. In 2000 the adaptive control paper he coauthored with G. C. Goodwin and P. J. Ramadge (IEEE Transactions on Automatic Control, 1980) was recognized by the IEEE Control Systems Society as one of the 25 seminal control theory papers of the 20th century. He is a Life Fellow of the IEEE, and a Fellow of SIAM, the Institute of Mathematics and its Applications (UK) and the Canadian Institute for Advanced Research and is a member of Professional Engineers Ontario. He was elected to the Royal Society of Canada in 2003. In 2009 he received the IEEE Control Systems Society Bode Lecture Prize and in 2012 a Queen Elizabeth II Diamond Jubilee Medal. Peter Caines is the author of Linear Stochastic Systems, John Wiley, 1988, and is a Senior Editor of Nonlinear Analysis – Hybrid Systems; his research interests include stochastic, mean field game, decentralized and hybrid systems theory,  together with their applications in a range of fields.
\end{IEEEbiography}

\end{document}